\documentclass[12pt]{amsart}

\usepackage{wasysym}
\usepackage{amsmath}
\usepackage{amssymb}
\usepackage{amsthm}
\usepackage{tikz}
\usetikzlibrary{matrix,arrows,backgrounds,shapes.misc,shapes.geometric,patterns,calc,positioning}
\usetikzlibrary{calc,shapes}
\usepackage{wrapfig}
\usepackage{epsfig}  
\usepackage[margin=1in]{geometry}
\usepackage{color}	 %
\input{xy}
\xyoption{poly}
\xyoption{2cell}
\xyoption{all}

\newcommand{\cala}{\mathcal{A}}

\newcommand{\calg}{\mathcal{G}}

\newcommand{\call}{\mathcal{L}}

\newcommand{\calgSW}{{}_{SW}\calg}
\newcommand{\calgNE}{\calg^{N\!E}}

\newtheorem{thm}{Theorem}[section]
\newtheorem{prop}[thm]{Proposition}
\newtheorem{lem}[thm]{Lemma}
\newtheorem{cor}[thm]{Corollary}

\theoremstyle{defn}
\newtheorem{definition}[thm]{Definition}
\newtheorem{example}[thm]{Example}

\theoremstyle{remark}

\newtheorem{remark}[thm]{Remark}

\numberwithin{equation}{section}

  \def\sgn{\operatorname{sgn}}

\newcommand{\za}{\alpha}
\newcommand{\zb}{\beta}
\newcommand{\zd}{\delta}

\newcommand{\ze}{\epsilon}

\newcommand{\zs}{\sigma}

\renewcommand{\sgn}{\mathop{sgn}}
\newcommand{\type}{\mathop{type}}

\newcommand{\cfa}{[a_1,a_2,\ldots,a_n]}
\newcommand{\cfb}{[b_1,b_2,\ldots,b_m]}
\newcommand{\Fa}{F_{a_1,a_2,\ldots,a_n}}
\newcommand{\Fb}{F_{b_1,b_2,\ldots,b_m}}

\begin{document}
\title{Cluster algebras and Jones polynomials}
\author{Kyungyong Lee}
\address{Department of Mathematics, 
University of Nebraska-Lincoln,
Lincoln NE 68588-0130,
USA}
\email{klee24@unl.edu}
\author{Ralf Schiffler}\thanks{The first author was supported by the University of Nebraska--Lincoln, Korea Institute for Advanced Study, and NSA grant H98230-16-1-0059. The second author was supported by the NSF-CAREER grant  DMS-1254567, and by the University of Connecticut.}
\address{Department of Mathematics, University of Connecticut, 
Storrs, CT 06269-3009, USA}
\email{schiffler@math.uconn.edu}

%
%
%
\begin{abstract} We present a new and very concrete connection between cluster algebras and knot theory. This connection is being made via  continued fractions and snake graphs.

 It is known that the class of 2-bridge knots and links is parametrized by continued fractions, and it has recently been shown that one can associate to each continued fraction a snake graph, and hence a cluster variable in a cluster algebra. We show that up to normalization by the leading term the Jones polynomial of the 2-bridge link is equal to the specialization of this cluster variable obtained by setting all initial cluster variables to 1 and specializing the initial principal coefficients of the cluster algebra as follows $y_1=t^{-2}$ and  $ y_i=-t^{-1}$, for all $i> 1$.
 
As a consequence we obtain a direct formula for the Jones polynomial of a 2-bridge link as the numerator of a continued fraction of Laurent polynomials in $q=-t^{-1}$.
  We also obtain formulas for the degree and the width of the Jones polynomial, as well as for the first three and the last three coefficients.

Along the way, we also develop some basic facts about even continued fractions and construct their snake graphs. We show that the snake graph of an even continued fraction is ismorphic to the snake graph of a positive continued fraction if the continued fractions have the same value. We also give recursive formulas for the Jones polynomials.
\end{abstract}

 \maketitle

\setcounter{tocdepth}{2}
\tableofcontents


\section{Introduction}\label{sect 1}
Cluster algebras were introduced by Fomin and Zelevinsky in 2002 in the context of canonical bases in Lie theory \cite{FZ1}. Since then many people have worked on cluster algebras and the subject has developed into a theory of 
its own with deep connections to a number of research areas including representation theory of associative algebras and Lie algebras, combinatorics, number theory, dynamical systems, hyperbolic geometry, algebraic geometry and string theory.

In this paper, we develop a new connection between cluster algebras and knot theory. The key to this relation is the use of continued fractions in both areas. To every continued fraction, one can associate a knot or a 2-component link which is built up as a sequence of braids each of which corresponding to one entry of the continued fraction. The class of knots (and links) obtained in this way are called {\em 2-bridge knots}. They were first studied by Schubert in 1956 in \cite{Sc} and the relation to continued fractions goes back to Conway in 1970 in \cite{C}.

A significant part of knot theory is concerned with knot invariants, and one of the most important knot invariants is the {\em Jones polynomial}, introduced by Jones in 1984 in \cite{J}. The Jones polynomial of an oriented link is a Laurent polynomial in one variable $t^{\pm\, 1/2}$ with integer coefficients. 

A {\em cluster algebra} is a $\mathbb{Z}$-subalgebra of a field of rational functions in several variables. To define a cluster algebra, one constructs a possibly infinite set of generators called {\em cluster variables}, by a recursive method called mutation.  Each cluster variable is a  polynomial in two types of variables; the initial cluster variables $x_1^{\pm 1},x_2^{\pm 1},\ldots,x_N^{\pm 1}$ and the initial principal coefficients $y_1,y_2,\ldots,y_N$, see \cite{FZ1,FZ4}, and with positive integer coefficients, see \cite{LS4}.
The {\em $F$-polynomial} of the cluster variable is obtained by setting all initial cluster variables equal to 1.

The relation between cluster algebras and continued fractions was established recently by \c{C}anak\c{c}\i{} and the second author in \cite{CS4}. 
They showed that the set of positive continued fractions is in bijection with the set of abstract {\em snake graphs}. These are planar graphs that appeared naturally in the study of cluster algebras from surfaces in \cite{MS} (also \cite{Propp} for the special case of the once-punctured torus).
The bijection is such that the numerator of the continued fraction is equal to the number of perfect matchings of the snake graph. 

In \cite{MS,MSW}, the authors gave a combinatorial formula for the cluster variables for cluster algebras from surfaces, and in \cite{MSW2} this formula was used to construct canonical bases for these cluster algebras in the case where the surface has no punctures. For each cluster variable, the authors construct a  weighted snake graph  from  geometric information provided by the surface. 
The combinatorial formula expresses the cluster variables as a sum over all perfect matchings of that snake graph.

In \cite{CS4}, the authors used their new approach to snake graphs via continued fractions to give another formula for the cluster variables of a cluster algebra of surface type (with trivial coefficients) which expresses the cluster variable as the numerator of a continued fraction of Laurent polynomials. This formula was generalized by Rabideau \cite{R} to include cluster algebras with principal coefficients.

Our main result in this paper is that these cluster variables specialize to the Jones polynomials of 2-bridge links. 

In order to make this statement  precise, we must first ensure that the 2-bridge link as well as its orientation is uniquely determined by the continued fraction. We found that the most natural way to do so is to work with {\em even continued fractions} $\cfb$, where each entry $b_i$ is an even non-zero integer (possibly negative). This is no restriction, since every 2-bridge link is associated to an even continued fraction. Moreover, if $\cfb$ is an even continued fraction and $\cfa$ is a positive continued fraction such that both have the same value $\cfb=p/q=\cfa$, then the associated 2-bridge links are isotopic and the associated snake graphs are isomorphic.

To state our main theorem, recall from above that each cluster variable is a  polynomial in variables $x_1^{\pm 1},x_2^{\pm 1},\ldots,x_N^{\pm 1}, y_1,y_2,\ldots,y_N$, and that the Jones polynomial is a  polynomial in one variable $t^{\pm\, 1/2}$.
Denote by $\Fb$ the specialization of the  cluster variable of the continued fraction $\cfb$ obtained by setting $x_1=x_2=\cdots=x_N=1$, $y_2=y_3=\cdots=y_N=-t^{-1}$ and $y_1=t^{-2}$.
Then our main result is the following.

\begin{thm}
 \label{thm main intro} (Theorem \ref{thm main})
 Let $\cfb$ be an even continued fraction, let $V_{\cfb}$ be the Jones polynomial of the corresponding 2-bridge link, and let $\Fb$ be the specialized cluster variable. Then
\begin{equation}
 V_{\cfb}= \zd \,t^j \ \Fb,
\end{equation}
 \smallskip

\noindent where $\zd=\pm 1$ and $j=  \sum_{i=1}^m  \max\left( (-1)^{i+1} b_i + \frac{\textup{sign}(b_ib_{i-1})}{2}\ ,\ -\frac{1}{2}\right).$
\end{thm}
 We also give a simple formula for the sign $\zd$. Thus the Jones polynomial is completely determined by the cluster variable, and hence this theorem establishes a direct connection between knot theory on the one hand and cluster algebras on the other. This also shows that  the Jones polynomial of a 2-bridge knot  is an alternating sum, which is not true for arbitrary knots. Another direct consequence of the theorem is that the difference between the highest and the lowest degree in the Jones polynomial is equal to   the sum of the entries in the positive continued fraction $a_1+a_2+\cdots+a_n$.

We also obtain the following direct formula for the Jones polynomial in terms of the continued fraction. In the statement, we use  the notation  $[a]_q=1+q+q^2+\ldots+q^{a-1}$ for the $q$-analogue of the positive integer $a$,  and we let $q=-t^{-1}$, and $\ell_i=a_1+a_2+\cdots a_i$.

\begin{thm}(Theorem \ref{thm main2}) Let $\cfa$ be a positive continued fraction. Then up to normalization by its leading term, $V_{\cfa}$ is equal to the numerator of the following continued fraction 

\begin{itemize}
\item [{\rm(a)}] If $n$ is odd, 
 \[\Big[\, [a_1+1]_q -q\ ,\  [a_2]_q \,q^{-\ell_2} ,\ [a_3]_q\,q^{\ell_2+1} ,\ldots, [a_{2i}]_q\, q^{-\ell_{2i}},\ [a_{2i+1}]_q\, q^{\ell_{2i}+1} ,\ldots , \ [a_{n}]_q\, q^{\ell_{n-1}+1}\Big]. \]
 
\item [{\rm(b)}] If $n$ is even, 
\[q^{\ell_n}\,\Big[\, [a_1+1]_q -q\ ,\  [a_2]_q \,q^{-\ell_2} ,\ [a_3]_q\,q^{\ell_2+1} ,\ldots, [a_{2i}]_q\, q^{-\ell_{2i}},\ [a_{2i+1}]_q\, q^{\ell_{2i}+1} ,\ldots , \ [a_{n}]_q\, q^{-\ell_{n}}\Big]. \]
\end{itemize}
\end{thm}

This result is very useful for computations. We give several examples, the biggest being the Jones polynomial of a link with 20 crossings, corresponding to the continued fraction $[2,3,4,5,6]$.
The theorem also allows us to obtain direct formulas for the first 3 and the last 3 co coefficients of the Jones polynomials in Theorem \ref{thm coeff}.

As mentioned above, a main ingredient in the proofs is the use of even continued fractions. These were used in knot theory previously, see for example \cite{KM}, but not in cluster algebras. We also did not find an algebraic treatment of the topic in the literature, so we include the basic algebraic foundation of even continued fractions in section \ref{sect cf}, and we apply it to the study of snake graphs in section \ref{sect sg}. Our main result in section \ref{sect sg} is the following.
\begin{thm}(Theorem \ref{thm 1})
Given  a positive continued fraction $\cfa$ and an even continued fraction $\cfb$ such that both have the same value $p/q$, then the corresponding snake graphs are isomorphic.
  \end{thm}

We also study the Jones polynomials of 2-bridge links and develop recursive formulas which we need to prove our main results.

Let us point out that  Hikami and Inoue have obtained a realization of knots in a cluster algebra of a disc with several punctures  using the braid presentation of the knot in \cite{HI,HI2}. Their approach as well as their results are very different from ours. Let us also mention the work of Shende, Treumann and Williams, who relate cluster algebras to Legendrian knots \cite{STW}. Also their work is very different from ours.

The paper is organized as follows. In section \ref{sect cf}, we review positive continued fractions and develop basic results for even continued fractions. We introduce snake graphs of even continued fractions in section \ref{sect sg} and compare them to snake graphs of positive continued fractions. Section \ref{sect knots} contains a short overview of 2-bridge links. We also fix conventions for the links and their orientations in that section. In section \ref{sect Jones}, we define the Jones polynomial and develop recursive formulas for the Jones polynomial of 2-bridge links from the definition. We also compute the degree of the Jones polynomial  and the sign of the leading coefficient. Section \ref{sect F} contains a brief review of cluster algebras, especially the formula for the (specialized) cluster variables in terms of continued fractions. We also prove recursive formulas for the $F$-polynomials in this section.
Our main results are stated and proved in section \ref{sect main}. This section also contains several examples. 

We would like to thank Michael Shapiro for helpful comments.

\section{Positive continued fractions vs even continued fractions}\label{sect cf} In this section, we list a few results on continued fractions that we will need later. 
For a standard introduction, we refer to  \cite[Chapter 10] {HW}. For a more extensive treatment of continued fractions see \cite{Perron}. For the  relation between continued fractions and cluster algebras see \cite{CS4}.

In this paper a \emph{continued fraction} is an expression of the form
\[[a_1,a_2,\ldots,a_n]= a_1+\cfrac{1}{a_2+\cfrac{1}{a_3+\cfrac{1}{\ddots +\cfrac{1}{a_n}}}}\]
where the $a_i$ are  integers (unless stated otherwise) and  $a_n\ne 0$. We will be mainly concerned with the following two types of continued fractions. A continued fraction is called \emph{positive} if each $a_i$ is a positive integer, and 
it is called \emph{even} if each $a_i$ is a nonzero even integer. 

For example, the rational number $\frac{27}{10}$ can be represented as a positive continued fraction as well as an even continued fraction as follows.
\[ \frac{27}{10}=[2,1,2,3]=[2,2,-2,4].\]
Indeed, the first expression is obtained by the Euclidean division algorithm
\[
\begin{array}{rcl}
 27&=&2\cdot 10 +7 \\
 10&=&1\cdot 7 +3 \\ 
7&=&2\cdot 3 +1 \\
3&=&3\cdot 1 ,\\
\end{array}\]
and the second expression is obtained by the following division algorithm
\[
\begin{array}{rcl}
 27&=&2\cdot 10 +7 \\
 10&=&2\cdot 7 +(-4) \\ 
7&=&(-2) \cdot (-4) + (-1) \\
-4&=&4\cdot (-1). \\
\end{array}\]

We often use the notation $N\cfa$ for the numerator of the continued fraction. Thus $N[2,1,2,3]=27$.

\begin{lem}\label{lem 0}
 Let $p>q>0$ be integers. \begin{itemize}
\item [\rm{(a)}] There exist unique integers $a,r$ such that $p=aq+r$, with $a>0$ and $0\le r<q$.  
\item [\rm{(b)}] There exist unique integers $b,s$ such that $p=bq+s$, with $b\in 2\mathbb{Z}\setminus\{0\}$ and $-q\le s<q$.
\end{itemize}
Moreover, if $pq$ is even then so is $qs$.
\end{lem}
\begin{proof}
 Part (a) is the Euclidean division algorithm in $\mathbb{Z}$. To prove (b), we let 
 $b=a$ and $s=r$ if $a $ is even, and we let $b=a+1$ and $s=r-q$ if $a$ is odd.
 Then $p=bq+s$ and the inequalities $0\le r<q$ imply the inequalities $-q\le s<q$.
 To prove uniqueness, suppose we have another pair $b',s'$ satisfying the required property. Then, since $bq+s=b'q+s'$, we have $|b-b'| q=|s'-s| < 2q$, and since $b-b'$ is even, it follows that $b=b'$ and $s=s'$.
\end{proof}
\begin{example}\label{ex 1} The even continued fraction of
 $p/(p-1)$ is $[2,-2,2,-2,\ldots, \pm 2]$ with a total of $m=p-1$ coefficients. 
 Thus $2/1=[2],3/2=[2,-2], 4/3=[2,-2,2], 5/4=[2,-2,2,-2]$.
\end{example}
\begin{prop}
 \label{prop 01}
 Let $p,q$ be relatively prime integers with $p>q>0$. Then
\begin{itemize}
\item [\rm{(a)}] $p/q$ has a positive continued fraction expansion $p/q=\cfa$ which is unique up to replacing the coefficient $a_n$ by the two coefficients $a_n-1,1$. 
\item [\rm{(b)}] if $p $ or $q$ is even, then $p/q$ has a unique even continued fraction expansion.
\item [\rm{(c)}]
 if $p$ and $q$ are both odd then $p/q$ does not have an even continued fraction expansion.
\end{itemize}
\end{prop}
\begin{proof}
 Part (a) is well known, see for example \cite[Theorem 162]{HW}.  We prove part (b) by induction on $p$. If $p=2$ then $q=1$ and $p/q=2=[2]$. To show that this expansion is unique, suppose that $2=\cfb$ with $m>1$. Then $2=b_1+1/[b_2,\ldots,b_m]$, which implies that $-1\le 2-b_1\le 1$. Since $b_1 $ is even, we conclude $b_1=2$ and thus $1/[b_2,\ldots,b_m]=0$ which is impossible.

 
 Now suppose that $p\ge 2$. Lemma  \ref{lem 0} 
 implies the existence of unique $b_1, s$ with $p=b_1 q+s$ with $b_1$ even and $-q\le s<q$, and $qs$ is even. By induction, we may assume that  $q/s$ has a unique even continued fraction expansion $q/s=[b_2,\ldots,b_m]$. Thus $p/q=b_1+s/q=\cfb$.
To show uniqueness, suppose $p/q=[b_1',b_2',\ldots,b_m']=b_1'+ 1/[b_2',\ldots,b_m']$. Then
$p=b_1'q + q/[b_2',\ldots,b_m']$ with $-q< q/[b_2',\ldots,b_m'] < q$, and then the uniqueness in Lemma \ref{lem 0} implies that $b_1'=b_1$ and $q/[b_2',\ldots,b_m']= s$.  Now the statement follows from the uniqueness of the expansion $q/s=[b_2,\ldots,b_m]$.

To show part (c), we prove by induction that if $\cfb=p/q$ with $b_i $ even, then one of $p$ or $q$ is even. Clearly $[b_1]=b_1$ is even.  For the induction step, suppose that $p/q=\cfb=b_1+1/[b_2,\ldots, b_m]$, and let $[b_2,\ldots, b_m]=q/r$. Then one of  $q$ or $r$ is even, by induction. If $q$ is even we are done, and if $r$ is even then $p/q=b_1+r/q$ implies that $p$ is even.
\end{proof}

\begin{prop}
 \label{prop 1}
 Let $p>q>1$ be relatively prime integers such that $p$ or $q$ is even. Denote by $p/q=\cfa$ a positive continued fraction expansion and by $p/q=\cfb$ the even continued fraction expansion. Let $q/r=[a_2,\ldots,a_n]$. Then 
 \[ [b_2,\ldots,b_m] = \left\{\begin{array}{ll} 
 q/r &\textup{if $a_1 $ is even;}\\
 -q/(q-r) &\textup{if $a_1 $ is odd.}
 \end{array}\right.\]
\end{prop}
\begin{proof} By Lemma \ref{lem 0}, we have $p=a_1q+r=b_1q+s$ and $[b_2,\ldots,b_m]=q/s$.
 If $a_1$ is even, then $a_1=b_1$, $r=s$. If $a_1$ is odd, then 
$b_1=a_1+1$ and $s=r-q$.
\end{proof}

\begin{prop}
 \label{prop 2}
 Let $p/q=\cfb$ with nonzero even integers $b_i$. Then $p$ is odd if and only if $m$ is even. 
\end{prop}
\begin{proof}
 If $m=1$ then $p=b_1$ is even. If $m=2$ then $p=b_1b_2+1$ is odd. Now suppose $m>2$. Then the recursion relation for the convergents of the continued fraction \cite[Theorem 149]{HW} gives 
 
\begin{equation}
 \label{def N}
\cfb =b_m N[b_1,b_2,\ldots,b_{m-1}] + N[b_1,\ldots, b_{m-2}]. 
\end{equation}
The first summand on the right hand side is always even since $b_m$ is even. Thus the parity of $N\cfb$ is the same as the parity of $N[b_1,b_2,\ldots b_{m-2}]$, and the result follows by induction.
\end{proof}

\smallskip 

The following definition will be crucial in the whole paper.
\begin{definition}
 \label{defsignb} Let $\cfb$ be an even continued fraction.
 \begin{itemize}
\item[{\rm (a)}]
The \emph{sign sequence} $\sgn\cfb$ of $\cfb$ is  the sequence
\\
{\small \begin{equation*} 
\begin{array}{cccccccc}
  ( \underbrace{ \sgn(b_1),\ldots,\sgn(b_1)},&  \underbrace{ -\sgn(b_2),\ldots,-\sgn(b_2)},& \ldots  &
  \underbrace{(-1)^{m+1}\sgn(b_m),\ldots,(-1)^{m+1}\sgn(b_m))}. \\
| b_1| & |b_2| &\ldots&|b_m| \\
\end{array} 
\end{equation*}
}

\smallskip

\noindent Thus the first $|b_1|$ entries of $\sgn\cfb$ are equal to the sign of $b_1$, the next $|b_2|$ entries are equal to the sign of $-b_2$ and so on.
\\
%

\item[{\rm (b)}]
The  \emph{type sequence $\type\cfb$ of $\cfb$} is the sequence
\[
(\sgn(b_1),-\sgn(b_2),\ldots,(-1)^{m+1}\sgn(b_m)).\]
\end{itemize}
\end{definition}

 For example 
\[\begin{array}
 {rclrrcl}\sgn[2,4,2]&=&(+,+,\,-,-,-,-,\,+,+),&\quad&\type[2,4,2]& =& (+,-,+) 
\\
 \sgn[2,2,-2,4]&=&(+,+,\,-,-,\,-,-,-,-,\,-,-),
 &
&\type[2,2,-2,4]&=&(+,-,-,-).
\end{array}
\]

\section{Snake graphs}\label{sect sg}
In this section, we review the construction of a snake graph from a positive continued fraction, and  we  introduce the new construction of a snake graph from an even continued fraction. We show that the two constructions are compatible in Theorem~\ref{thm 1}.

\subsection{Abstract snake graphs} 
Abstract snake graphs   have been introduced and studied \cite{CS,CS2,CS3} motivated by the snake graphs  appearing in the combinatorial formulas for elements in cluster algebras of surface type  in \cite{Propp,MS,MSW,MSW2}. 
In this section we review the definition. Throughout
we fix the standard orthonormal basis of the plane.

 A {\em tile} $G$ is a square in the plane whose sides are parallel or orthogonal  to the elements in    the  fixed basis. All tiles considered will have the same side length.
 We consider a tile $G$ as  a graph with four vertices and four edges in the obvious way. A {\em snake graph} $\calg$ is a connected planar graph consisting of a finite sequence of tiles $G_1,G_2,\ldots, G_d$ with $d \geq 1,$ such that
$G_i$ and $G_{i+1}$ share exactly one edge $e_i$ and this edge is either the north edge of $G_i$ and the south edge of $G_{i+1}$ or the east edge of $G_i$ and the west edge of $G_{i+1}$,  for each $i=1,\dots,d-1$.
An example is given in Figure \ref{signfigure}.
\begin{figure}
\begin{center}
  {\tiny \scalebox{0.9}{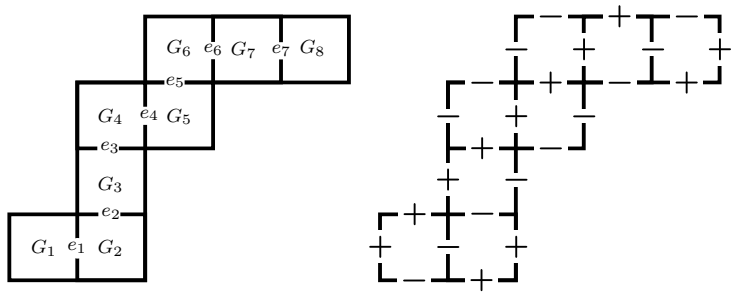}}
 \caption{A snake graph with 8 tiles and 7 interior edges (left);
 a sign function on the same snake graph (right)} 
 \label{signfigure}
\end{center}
\end{figure}
The graph consisting of two vertices and one edge joining them is also considered a snake graph.

The $d-1$ edges $e_1,e_2, \dots, e_{d-1}$ which are contained in two tiles are called {\em interior edges} of $\calg$ and the other edges are called {\em boundary edges.}  
We will always use the natural ordering of the set of interior edges, so that $e_i$ is the edge shared by the tiles $G_i$ and $G_{i+1}$.

We denote by  $\calgSW$ the 2 element set containing the south and the west edge of the first tile of $\calg$ and by $\calgNE$ the 2 element set containing the north and the east edge of the last tile of $\calg$. 
 If $\calg$ is a single edge, we let $\calgSW=\emptyset$ and $\calgNE=\emptyset$. 

A snake graph $\calg$ is called {\em straight} if all its tiles lie in one column or one row, and a snake graph is called {\em zigzag} if no three consecutive tiles are straight.
 We say that two snake graphs are \emph{isomorphic} if they are isomorphic as graphs.

A {\em sign function} $f$ on a snake graph $\calg$ is a map $f$ from the set of edges of $\calg$ to $\{ +,- \}$ such that on every tile in $\calg$ the north and the west edge have the same sign, the south and the east edge have the same sign and the sign on the north edge is opposite to the sign on the south edge. See Figure \ref{signfigure} for an example.
%
%

Note that on every snake graph  there are exactly two sign functions. A snake graph is  determined up to symmetry by its sequence of tiles together with a sign function on its interior edges.

\subsection{Snake graphs from positive continued fractions}\label{sect sgb}
In this section, we recall a construction from \cite{CS4} which associates a snake graph to a positive continued fraction.

Given a snake graph $\calg$ with $d$ tiles, we consider the set of interior edges $\{e_1,e_2,\ldots,e_{d-1}\} $ and we let $e_0$ be  one of the two boundary edges in $\calgSW$, and we let $e_d$ be one of the two boundary edges in $\calgNE$. Let $f$ be a sign function on $\calg$ and consider the {\em sign sequence}  
$(f(e_0),f(e_1),\ldots,f(e_d)) .$ Let $-\ze\in\{\pm\}$ be the first sign $f(e_0)$ in this sequence, and define a continued fraction $[a_1,a_2,\ldots,a_n]$ from the sign sequence as follows 
\begin{equation}
 \label{eqsign1} 
\begin{array}{cccccccc}
  ( \underbrace{ -\ze,\ldots,-\ze},&  \underbrace{ \ze,\ldots,\ze},&  \underbrace{ -\ze,\ldots,-\ze},& \ldots,&  \underbrace{\pm\ze,\ldots,\pm\ze}) . \\
 a_1 & a_2 & a_3&\ldots&a_n
\end{array} 
\end{equation}
Thus $a_1$ is the number of entries before the first sign change, $a_2 $ is the number of entries between the first and the second sign change and so on. See the top row of Figure \ref{match4} for examples.

Conversely, 
to every positive integer $a_i$, we associate the snake graph
$\calg[a_i]$ consisting of $a_i-1$ tiles and $a_i-2$ interior edges all of which have the same sign. Thus $\calg[a_i]$ is a  zigzag snake graph, meaning that no three consecutive tiles lie in one row or in one column.

If $\cfa$ is a positive continued fraction, its snake graph $\calg\cfa$ is defined as the unique snake graph with $d=a_1+a_2+\cdots +a_n-1$ tiles such that the signs of the $d-1$ interior edges form the following sign sequence.
\begin{equation}
 \label{eqsign} 
\begin{array}{cccccccc}
  ( \underbrace{ -\ze,\ldots,-\ze},&  \underbrace{ \ze,\ldots,\ze},&  \underbrace{ -\ze,\ldots,-\ze},& \ldots,&  \underbrace{\pm\ze,\ldots,\pm\ze}) . \\
 a_1-1 & a_2 & a_3&\ldots&a_n-1
\end{array} 
\end{equation}
Furthermore the continued fraction determines a choice of edges $e_0$ in $\calgSW$ and $e_d$ in $\calgNE$ via the sign condition. 

For example, the snake graph of $27/10=[2,1,2,3]$ is shown in Figure \ref{fig2710}.
Note that between any two consecutive zigzag graphs $\calg[a_i]$ and $\calg[a_{i+1}]$ there is exactly one tile that is not part of the zigzag graphs.

\begin{figure}
\begin{center}
  { \scalebox{1}{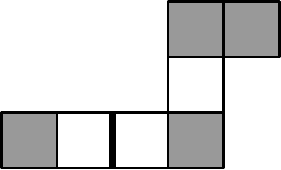}}
 \caption{The snake graph $\calg[2,1,2,3]$. The zigzag subgraphs $\calg[2],\calg[2],\calg[3]$   are shaded and the subgraph $
\calg[1]$ consist of the interior edge shared by  the second and third tile.} 
 \label{fig2710}
\end{center}
\end{figure}

\medskip
Snake graphs were introduced in \cite{MS,MSW} to describe cluster variables in cluster algebra of surface type. The cluster variable is given as a sum over all perfect matchings of the snake graph. Therefore the main interest in the construction of snake graphs from continued fractions stems from the following result.
\begin{thm}
 \cite{CS4}  
 \[\cfa =\frac{m(\calg\cfa)}{m(\calg[a_2,\ldots,a_n])},\] 
 where $m(\calg)$ is the number of perfect matchings of the graph $\calg$. Moreover the fraction on the right hand side is reduced.
 
 In particular, the numerator $N\cfa$ of the continued fraction is the number of perfect matchings of the snake graph $\calg\cfa$.
\end{thm}

\subsection{Snake graphs from even continued fractions}
In this subsection, we extend the construction of \cite{CS4} to \emph{even} continued fractions.  Let $\cfb$ be an even continued fraction. Again we form the zigzag graphs  $\calg[\,|b_1|\,],\ldots,\calg[\,|b_m|\,]$ of the absolute values of the $b_i$, but now we join them according to the following rule, see Figure \ref{figgluing}.

\begin{figure}
\begin{center}
  { \scalebox{1}{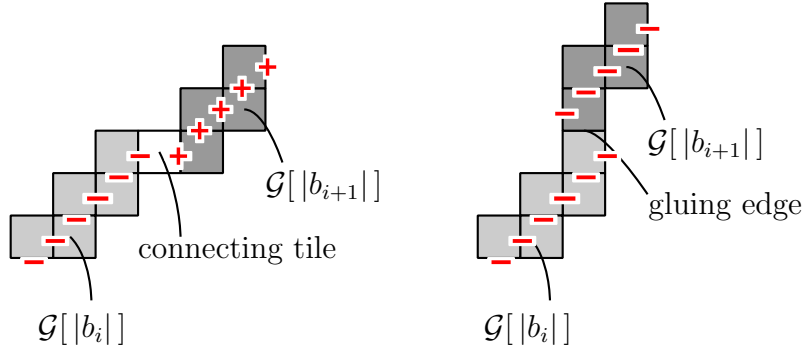}}
 \caption{Joining the zigzag graphs. On the left $b_i$ and $b_{i+1}$ have the same sign, and on the right they have opposite signs.} 
 \label{figgluing}
\end{center}
\end{figure}

Without loss of generality, assume that the last tile of $\calg[\,|b_i|\,]$ is north of the second to last tile, then the first tile of $\calg[\,|b_{i+1}|\,]$ is south of its second tile.
If  $b_i$ and $b_{i+1}$ have the same sign then we join the two graphs by drawing two horizontal edges that form a connecting tile together with the east edge of the last tile of $\calg[\,|b_i|\,]$ and the west edge of the first tile of $\calg[\,|b_{i+1}|\,]$. This case is illustrated in the left picture in Figure \ref{figgluing}.
If $b_i$ and $b_{i+1}$ have opposite signs then we glue the two graphs by identifying the north edge of the last tile of $\calg[\,|b_i|\,]$  with the south edge of the first tile of   $\calg[\,|b_{i+1}|\,]$. This case is illustrated in the right picture in Figure \ref{figgluing}.
Thus if  $b_i$ and $b_{i+1}$ have the same sign then the sign function on the interior edges changes from the subgraph $\calg[\,|b_{i}|\,]$ to $\calg[\,|b_{i+1}|\,]$. On the other hand, if  $b_i$ and $b_{i+1}$ have opposite signs then the sign function stays the same from  $\calg[\,|b_{i}|\,]$ to $\calg[\,|b_{i+1}|\,]$.

The following lemma follows directly from the construction.
\begin{lem}
 \label{lem extra}
 \begin{itemize}
\item[{\rm (a)}] $\calg[-b_1,-b_2,\ldots,-b_m] \cong \calg \cfb$.
\item[{\rm (b)}] $\calg[b_1,b_2]\cong\left\{\begin{array}{ll} \calg[\,|b_1|\,,\,|b_2|\,] &\textup{if $b_1b_2>0$;}\\[8pt]
\calg[\,|b_1|-1,1,\,|b_2|-1] &\textup{if $b_1b_2<0$.}
\end{array}\right.$ \\
\item[{\rm (c)}] $\calg[2,-2,2,-2,\ldots] $ is a zigzag snake graph.
\item[{\rm (d)}] $\calg[2,2,-2,-2,2,2,-2,-2,\ldots]$ is a straight snake graph. 
\end{itemize}
\end{lem}

\subsubsection{The sign sequence of  a snake graph of an even continued fraction} We choose our sign function $f$ such that the south edge of the first tile $e_0$ has the same sign as $b_1$. Thus $f(e_0)=\sgn(b_1)$. 
We want to define a  sequence of edges  by $e_0'=e_0,e_1',\ldots,e'_{
\zb}$ in $\calg\cfb$ that realize the sign sequence $\sgn\cfb$ for $\cfb$ introduced in Definition~\ref{defsignb}; here $\zb=|b_1|+\cdots |b_m|$.
The first edge in the sequence is $e_0$. On the subgraph $\calg[\,|b_1|\,]$, we choose the interior edges $e'_1,\ldots,e'_{|b_1|-1}$ and the unique edge $e'_{|b_1|}\in\calg[\,|b_1|\,]^{N\!E}$ such that $f(e'_{|b_1|})=\sgn(b_1)$. Similarly, on the subgraph  $\calg[\,|b_i|\,]$, we choose the interior edges  and the two unique edges in ${}_{SW}\calg[\,|b_i|\,]$ and $\calg[\,|b_i|\,]^{N\!E}$  whose signs are $(-1)^{i+1}\sgn(b_i).$ See Figure \ref{fig signb} and the bottom row of Figure \ref{match4} for examples.
 Thus we have the following equality of sign sequences 
\begin{equation}
 \label{eqsignb} (f(e_0'),f(e_2'),\ldots,f(e'_{\zb}))=\sgn\cfb.
\end{equation}

\begin{figure}
\begin{center}
  {\small \scalebox{1}{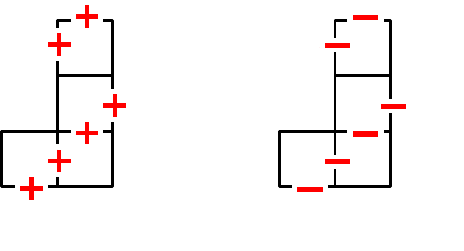}}
 \caption{The edges $e_0',e_1',\ldots,e
 _{\zb}$ carry the signs of $\cfb$. The sign of $e_0'$ is equal to the sign of $b_1$.} 
 \label{fig signb}
\end{center}
\end{figure}

\begin{figure}
\begin{center}
 \small { \scalebox{0.7}{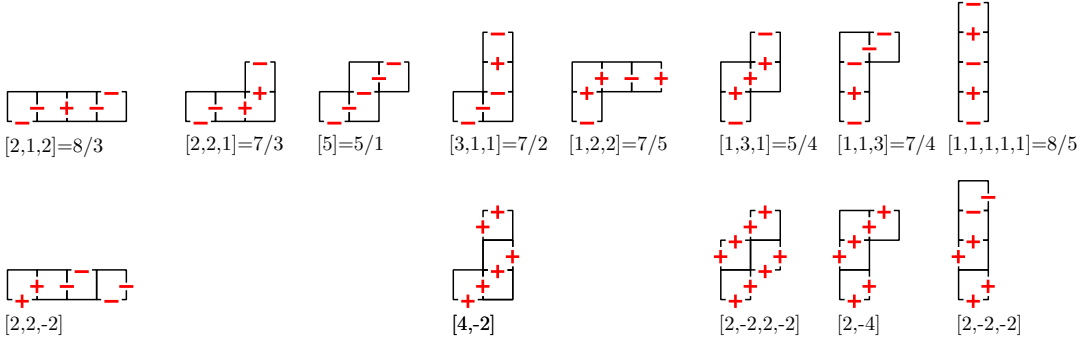}}
 \caption{The top row shows a complete list of snake graphs with 4  tiles together with their positive continued fractions. Those which admit even continued fractions are repeated in the bottom row. The signed edges indicate the sign sequence of the continued fraction.}
 \label{match4}
\end{center}
\end{figure}

\subsection{Correspondence between the two constructions}
We are ready for the  main result of this section.

\begin{thm}
 \label{thm 1}
 Let $\cfa=p/q$  be a positive continued fraction with $a_1>1$,  and  let $\cfb$ be an even continued fraction such that  $\cfb=p/q$ or $p/(p-q)$. Then the following  snake graphs are isomorphic
 \[\calg\cfa\cong\calg[1,a_1-1,a_2,\ldots,a_n]\cong\calg\cfb.\]
 In particular, the number of perfect matchings of $\calg\cfb$ is equal to $p$.
\end{thm}
\begin{proof} The first isomorphism follows directly from the construction of the snake graphs. Indeed the difference between the two snake graphs consists only in the choice of the boundary edge $e_0$.

To prove the second isomorphism, we proceed by induction on $n$. If $n=1$ then  $\calg[a_1]$ is a zigzag graph with $a_1-1$ tiles. If $a_1$ is even then $b_1=a_1$ and we are done. If $a_1$ is odd, then the even continued fraction  cannot be equal to $a_1/1$ and therefore we must have $\cfb=a_1/(a_1-1)$.
 We have seen in Example \ref{ex 1} that $\cfb=[2,-2,2,-2,...\pm2]$ with $m=a_1-1$.  Since the signs in this continued fraction are alternating, the corresponding snake graph $\calg\cfb$ is a zigzag snake graph with $a_1-1$ tiles, thus $\calg[a_1]\cong\calg\cfb$. 
    
Suppose now that $n>1$.
Since $\calg\cfa\cong\calg[1,a_1-1,a_2,\ldots a_n]$ and  the continued fraction $ [1,a_1-1,a_2,\ldots a_n]=p/(p-q)$, we may assume without loss of generality that $\cfa=\cfb=p/q$. Let $q/r $ denote the value of $[a_2,\ldots, a_n]$.
From Proposition \ref{prop 1} we know that  $[b_2,\ldots,b_m] = q/r$ or $-q/(q-r)$. Thus by induction, we may assume that $\calg[a_2,\ldots,a_n] \cong\calg[b_2,\ldots,b_m] $.

The graph $\calg\cfa$ is obtained from the subgraphs $\calg[a_1]$ and $\calg[a_2,\ldots,a_n]$ as described above by inserting two horizontal edges to form the connecting tile $G_{a_1}$. Since $\calg[a_1]$ has $a_1-1$ tiles, the connecting tile $G_{a_1}$ is the $a_1$-th tile of the snake graph $\calg\cfa$. 

If $a_1$ is even,  then $a_1=b_1$, and then $b_1$ and $b_2$ have the same sign by the division algorithm. Thus $\calg[a_2,\ldots,a_n]\cong\calg[b_2,\ldots,b_m]$ and the gluing with $\calg[a_1]=\calg[b_1]$ which yields $\calg\cfa$ respectively $\calg\cfb$ is the same. This shows the result if $a_1$ is even.

Suppose now that $a_1 $ is odd. Then $b_1=a_1+1>0$ and $b_2<0$. 
Now,  $\calg[b_1]$ has one more tile  than $\calg[a_1]$, however, since $b_1 $ and $b_2$ have opposite signs, the procedure by which we glue $\calg[b_1]$ with $\calg[b_2,\ldots,b_m]$ yields the same snake graph as the gluing of $\calg[a_1]$ with $\calg[a_2,\ldots,a_n]$. Thus $\calg\cfa$ is isomorphic to $\calg\cfb$. 
\end{proof}

\begin{cor}\label{cor 3.3}
The  number of tiles of $\calg\cfb $ is equal to \[\sum_{i=1}^m|b_i| - 1- \textup{(number of sign changes in $b_1,b_2,\ldots,b_m$)}.\]
\end{cor}

\begin{proof}
 If $m=1$ then $\calg[b_1]$ has $b_{1}-1$ tiles. For $m>1$, we have seen that $\calg\cfb$ is obtained by gluing $\calg[b_1]$ to $\calg[b_2,\ldots,b_m]$. By induction we may assume that the number of tiles in  $\calg[b_2,\ldots,b_m]$ is $\sum_{i=2}^{m}|b_i| - 1- $(number of sign changes in $b_2,\ldots,b_m$), and gluing $\calg[b_1]$ to it will increase this number by $b_1$, if $b_1$ and $b_2 $ have the same sign, and by $b_1-1 $, if the signs of $b_1$ and $b_2 $ are opposite. 
This completes the proof.
\end{proof}
\begin{remark}
 Theorem \ref{thm 1} and its corollary also hold more generally for continued fractions with entries $b_i \in \mathbb{Z}\setminus\{0,-1\}$, but we will not need this here.
\end{remark}

\section{Two-bridge knots from continued fractions}\label{sect knots}
In this  section, we briefly review two-bridge knots. These knots have been studied first by Schubert in 1956
\cite{Sc}. For a detailed introduction see for example \cite{KL}.

A \emph{knot} is a subset  of $\mathbb{R}^3$ that is homeomorphic to a circle. A \emph{link with $r$ components} is a subset  of $\mathbb{R}^3$  that is homemorphic to a disjoint union of $r$ circles. Thus a knot is a link with one component. A knot is \emph{alternating} if the crossings alternate between over and under when traveling along a strand.

To every rational number $p/q\ge 1$ one can associate a link $C(p/q)$ or $C\cfa$ using the positive continued fraction expansion $p/q=\cfa$. 

\begin{example}
$ 27/10=[2,1,2,3]$ corresponds to the knot shown on the left in Figure \ref{knot2123}.
\end{example}

\begin{figure}
\begin{center}
\scalebox{0.6}  {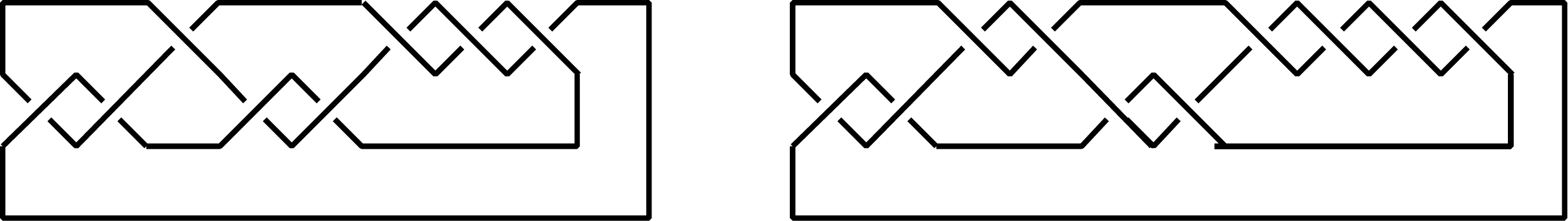}
 \caption{Two isotopic knots $C(27/10)=C[2,1,2,3]\cong C[2,2,-2,4]$. The knot on the left has a positive continued fraction $[2,1,2,3]$ and is alternating. The knot on the right has an even continued fraction $[2,2,-2,4]$ and its diagram is not alternating from the second to the third braid and from the third to the forth braid.  }
 \label{knot2123}
\end{center}
\end{figure}

The link consists of $n$ pieces, each of which is a 2-strand braid with $a_i$ crossings, where $i=1,2,\ldots,n$. These pieces are joined in such a way that the link is alternating, see Figure~\ref{knotschema}.
\begin{figure}
\begin{center}
  \scalebox{0.8}{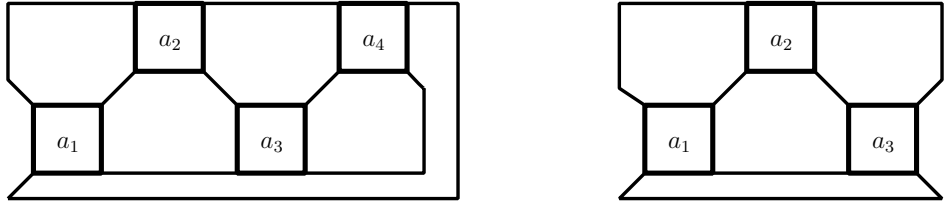}
 \caption{2-Bridge link diagrams; $n=4$ on the left, $n=3$ on the right; each box represents a 2-strand braid with $a_i$ crossings.}
 \label{knotschema}
\end{center}
\end{figure}
It is well known that the link $C(p/q)$ is a knot if $p$ is odd and it is a link with exactly two connected components if $p$ is even \cite{Sc}.  Moreover, in the case where $p$ is odd, the two knots $C(p/q)$ and $C(p/(p-q))$ are isotopic. Indeed this follows simply from the fact that the continued fractions are related as follows $p/q=\cfa$ and $p/(p-q)=[1,a_1-1,a_2,\ldots,a_n]$.
Thus when we are considering the knot or link $C\cfa$, we may always assume that $a_1>1$. 

On the other hand, by Proposition \ref{prop 01}, at least one of the rational numbers $p/q $ and $p/(p-q)$ can be represented by an even continued fraction $\cfb$. We can also construct a link $C\cfb$ from this even continued fraction, essentially in the same way, except that a sign change now means that the diagram is not alternating at that point. See for example the knot $C[2,2,-2,4]$ in Figure \ref{knot2123}.

It is known that if $p/q$ is equal to the positive continued fraction $\cfa$ and to the even continued fraction $\cfb$ then the corresponding links $C\cfa$ and $C\cfb$ are isotopic. See for example \cite{KM}.

\smallskip
\subsection{Orientation} One can orient the strand of a knot in one of two ways. In a link, one can orient each conponente in one of two ways. Fixing an orientation of the strands induces a sign on each crossing according to the following cases.
\[\begin{array}
{cccc}
\scalebox{0.6}{
\begingroup%
  \makeatletter%
  \providecommand\color[2][]{%
    \errmessage{(Inkscape) Color is used for the text in Inkscape, but the package 'color.sty' is not loaded}%
    \renewcommand\color[2][]{}%
  }%
  \providecommand\transparent[1]{%
    \errmessage{(Inkscape) Transparency is used (non-zero) for the text in Inkscape, but the package 'transparent.sty' is not loaded}%
    \renewcommand\transparent[1]{}%
  }%
  \providecommand\rotatebox[2]{#2}%
  \ifx\svgwidth\undefined%
    \setlength{\unitlength}{178.039744bp}%
    \ifx\svgscale\undefined%
      \relax%
    \else%
      \setlength{\unitlength}{\unitlength * \real{\svgscale}}%
    \fi%
  \else%
    \setlength{\unitlength}{\svgwidth}%
  \fi%
  \global\let\svgwidth\undefined%
  \global\let\svgscale\undefined%
  \makeatother%
  \begin{picture}(1,0.19321529)%
    \put(0,0){\includegraphics[width=\unitlength]{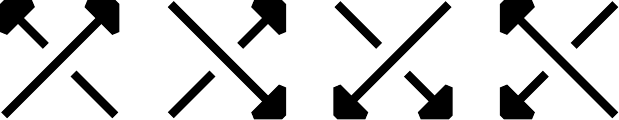}}%
  \end{picture}%
\endgroup%
}
& \qquad \textup{and}\qquad&
\scalebox{0.6}{
\begingroup%
  \makeatletter%
  \providecommand\color[2][]{%
    \errmessage{(Inkscape) Color is used for the text in Inkscape, but the package 'color.sty' is not loaded}%
    \renewcommand\color[2][]{}%
  }%
  \providecommand\transparent[1]{%
    \errmessage{(Inkscape) Transparency is used (non-zero) for the text in Inkscape, but the package 'transparent.sty' is not loaded}%
    \renewcommand\transparent[1]{}%
  }%
  \providecommand\rotatebox[2]{#2}%
  \ifx\svgwidth\undefined%
    \setlength{\unitlength}{178.039744bp}%
    \ifx\svgscale\undefined%
      \relax%
    \else%
      \setlength{\unitlength}{\unitlength * \real{\svgscale}}%
    \fi%
  \else%
    \setlength{\unitlength}{\svgwidth}%
  \fi%
  \global\let\svgwidth\undefined%
  \global\let\svgscale\undefined%
  \makeatother%
  \begin{picture}(1,0.1933305)%
    \put(0,0){\includegraphics[width=\unitlength]{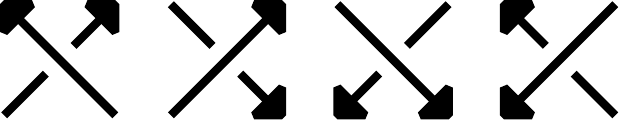}}%
  \end{picture}%
\endgroup%
}
 \\
 \textup{positive}&&  \textup{negative}
\end{array}\]

We shall always use the following convention for the orientations of the strands. If $\cfb$ is an even continued fraction then the first crossing in the $b_1$-braid is \scalebox{.4}{
\begingroup%
  \makeatletter%
  \providecommand\color[2][]{%
    \errmessage{(Inkscape) Color is used for the text in Inkscape, but the package 'color.sty' is not loaded}%
    \renewcommand\color[2][]{}%
  }%
  \providecommand\transparent[1]{%
    \errmessage{(Inkscape) Transparency is used (non-zero) for the text in Inkscape, but the package 'transparent.sty' is not loaded}%
    \renewcommand\transparent[1]{}%
  }%
  \providecommand\rotatebox[2]{#2}%
  \ifx\svgwidth\undefined%
    \setlength{\unitlength}{34.4bp}%
    \ifx\svgscale\undefined%
      \relax%
    \else%
      \setlength{\unitlength}{\unitlength * \real{\svgscale}}%
    \fi%
  \else%
    \setlength{\unitlength}{\svgwidth}%
  \fi%
  \global\let\svgwidth\undefined%
  \global\let\svgscale\undefined%
  \makeatother%
  \begin{picture}(1,0.98982558)%
    \put(0,0){\includegraphics[width=\unitlength]{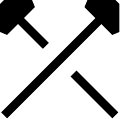}}%
  \end{picture}%
\endgroup%
} if $b_1>0$ and it is \scalebox{.4}{
\begingroup%
  \makeatletter%
  \providecommand\color[2][]{%
    \errmessage{(Inkscape) Color is used for the text in Inkscape, but the package 'color.sty' is not loaded}%
    \renewcommand\color[2][]{}%
  }%
  \providecommand\transparent[1]{%
    \errmessage{(Inkscape) Transparency is used (non-zero) for the text in Inkscape, but the package 'transparent.sty' is not loaded}%
    \renewcommand\transparent[1]{}%
  }%
  \providecommand\rotatebox[2]{#2}%
  \ifx\svgwidth\undefined%
    \setlength{\unitlength}{34.379488bp}%
    \ifx\svgscale\undefined%
      \relax%
    \else%
      \setlength{\unitlength}{\unitlength * \real{\svgscale}}%
    \fi%
  \else%
    \setlength{\unitlength}{\svgwidth}%
  \fi%
  \global\let\svgwidth\undefined%
  \global\let\svgscale\undefined%
  \makeatother%
  \begin{picture}(1,0.99041615)%
    \put(0,0){\includegraphics[width=\unitlength]{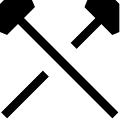}}%
  \end{picture}%
\endgroup%
} if $b_1<0$. In particular, the sign of the first crossing is equal to the sign of $b_1.$ The $|b_i|$ crossings in the $i$-th braid all have the same sign. We call the braid positive if these signs are $+$ and negative if they are $-$.

\begin{lem}
 \label{lem 3}
 With the conventions above, if $\cfb$ is an even continued fraction then the last crossing is of the following form.
\[\begin{array}{clccl} 
\scalebox{0.4}{
\begingroup%
  \makeatletter%
  \providecommand\color[2][]{%
    \errmessage{(Inkscape) Color is used for the text in Inkscape, but the package 'color.sty' is not loaded}%
    \renewcommand\color[2][]{}%
  }%
  \providecommand\transparent[1]{%
    \errmessage{(Inkscape) Transparency is used (non-zero) for the text in Inkscape, but the package 'transparent.sty' is not loaded}%
    \renewcommand\transparent[1]{}%
  }%
  \providecommand\rotatebox[2]{#2}%
  \ifx\svgwidth\undefined%
    \setlength{\unitlength}{34.4bp}%
    \ifx\svgscale\undefined%
      \relax%
    \else%
      \setlength{\unitlength}{\unitlength * \real{\svgscale}}%
    \fi%
  \else%
    \setlength{\unitlength}{\svgwidth}%
  \fi%
  \global\let\svgwidth\undefined%
  \global\let\svgscale\undefined%
  \makeatother%
  \begin{picture}(1,0.98982558)%
    \put(0,0){\includegraphics[width=\unitlength]{crossing2.pdf}}%
  \end{picture}%
\endgroup%
} &\textup{if  $m$ is even and $b_m>0$;} &\qquad&
\scalebox{0.4}{
\begingroup%
  \makeatletter%
  \providecommand\color[2][]{%
    \errmessage{(Inkscape) Color is used for the text in Inkscape, but the package 'color.sty' is not loaded}%
    \renewcommand\color[2][]{}%
  }%
  \providecommand\transparent[1]{%
    \errmessage{(Inkscape) Transparency is used (non-zero) for the text in Inkscape, but the package 'transparent.sty' is not loaded}%
    \renewcommand\transparent[1]{}%
  }%
  \providecommand\rotatebox[2]{#2}%
  \ifx\svgwidth\undefined%
    \setlength{\unitlength}{34.4bp}%
    \ifx\svgscale\undefined%
      \relax%
    \else%
      \setlength{\unitlength}{\unitlength * \real{\svgscale}}%
    \fi%
  \else%
    \setlength{\unitlength}{\svgwidth}%
  \fi%
  \global\let\svgwidth\undefined%
  \global\let\svgscale\undefined%
  \makeatother%
  \begin{picture}(1,0.98982558)%
    \put(0,0){\includegraphics[width=\unitlength]{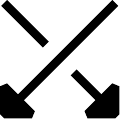}}%
  \end{picture}%
\endgroup%
} &\textup{if $m$ is odd and $b_m<0$};\\
\scalebox{0.4}{
\begingroup%
  \makeatletter%
  \providecommand\color[2][]{%
    \errmessage{(Inkscape) Color is used for the text in Inkscape, but the package 'color.sty' is not loaded}%
    \renewcommand\color[2][]{}%
  }%
  \providecommand\transparent[1]{%
    \errmessage{(Inkscape) Transparency is used (non-zero) for the text in Inkscape, but the package 'transparent.sty' is not loaded}%
    \renewcommand\transparent[1]{}%
  }%
  \providecommand\rotatebox[2]{#2}%
  \ifx\svgwidth\undefined%
    \setlength{\unitlength}{34.4bp}%
    \ifx\svgscale\undefined%
      \relax%
    \else%
      \setlength{\unitlength}{\unitlength * \real{\svgscale}}%
    \fi%
  \else%
    \setlength{\unitlength}{\svgwidth}%
  \fi%
  \global\let\svgwidth\undefined%
  \global\let\svgscale\undefined%
  \makeatother%
  \begin{picture}(1,0.98982558)%
    \put(0,0){\includegraphics[width=\unitlength]{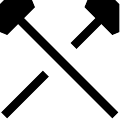}}%
  \end{picture}%
\endgroup%
} &\textup{if  $m$ is even and $b_m<0$;}& &
\scalebox{0.4}{
\begingroup%
  \makeatletter%
  \providecommand\color[2][]{%
    \errmessage{(Inkscape) Color is used for the text in Inkscape, but the package 'color.sty' is not loaded}%
    \renewcommand\color[2][]{}%
  }%
  \providecommand\transparent[1]{%
    \errmessage{(Inkscape) Transparency is used (non-zero) for the text in Inkscape, but the package 'transparent.sty' is not loaded}%
    \renewcommand\transparent[1]{}%
  }%
  \providecommand\rotatebox[2]{#2}%
  \ifx\svgwidth\undefined%
    \setlength{\unitlength}{34.4bp}%
    \ifx\svgscale\undefined%
      \relax%
    \else%
      \setlength{\unitlength}{\unitlength * \real{\svgscale}}%
    \fi%
  \else%
    \setlength{\unitlength}{\svgwidth}%
  \fi%
  \global\let\svgwidth\undefined%
  \global\let\svgscale\undefined%
  \makeatother%
  \begin{picture}(1,0.98982558)%
    \put(0,0){\includegraphics[width=\unitlength]{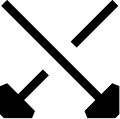}}%
  \end{picture}%
\endgroup%
} &\textup{if $m$ is odd and $b_m>0$}.
\end{array}\]
\end{lem}

\begin{proof}  Using our schematic illustration of the link, we see that the strand that enters the $b_i$-braid from southwest (respectively southeast)  must leave the $b_i$-braid towards southeast (respectively southwest), because $b_i$ is even. Similarly the strand entering the $b_i$-braid from the northeast (respectively northwest) will exit the braid to the northwest (respectively northeast). If $m$ is even, the orientation is as shown in Figure \ref{orientation}. 
\begin{figure}
\begin{center}
  \scalebox{0.8}{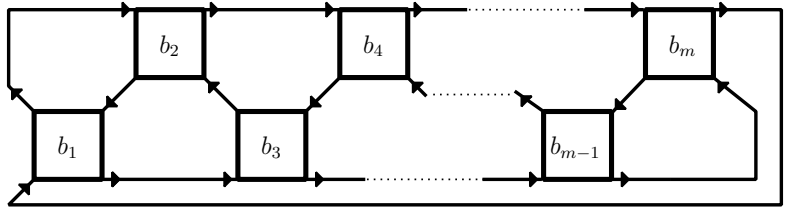}
 \caption{The orientation of the strands in $C\cfb$ when the $b_i$ are even and $m$ is even.}
 \label{orientation}
\end{center}
\end{figure}
Thus if $m$ is even then, since $b_m$  is even, the last crossing   is \scalebox{0.4}{
\begingroup%
  \makeatletter%
  \providecommand\color[2][]{%
    \errmessage{(Inkscape) Color is used for the text in Inkscape, but the package 'color.sty' is not loaded}%
    \renewcommand\color[2][]{}%
  }%
  \providecommand\transparent[1]{%
    \errmessage{(Inkscape) Transparency is used (non-zero) for the text in Inkscape, but the package 'transparent.sty' is not loaded}%
    \renewcommand\transparent[1]{}%
  }%
  \providecommand\rotatebox[2]{#2}%
  \ifx\svgwidth\undefined%
    \setlength{\unitlength}{34.4bp}%
    \ifx\svgscale\undefined%
      \relax%
    \else%
      \setlength{\unitlength}{\unitlength * \real{\svgscale}}%
    \fi%
  \else%
    \setlength{\unitlength}{\svgwidth}%
  \fi%
  \global\let\svgwidth\undefined%
  \global\let\svgscale\undefined%
  \makeatother%
  \begin{picture}(1,0.98982558)%
    \put(0,0){\includegraphics[width=\unitlength]{crossing2.pdf}}%
  \end{picture}%
\endgroup%
} if $b_m>0$ and  \scalebox{0.4}{
\begingroup%
  \makeatletter%
  \providecommand\color[2][]{%
    \errmessage{(Inkscape) Color is used for the text in Inkscape, but the package 'color.sty' is not loaded}%
    \renewcommand\color[2][]{}%
  }%
  \providecommand\transparent[1]{%
    \errmessage{(Inkscape) Transparency is used (non-zero) for the text in Inkscape, but the package 'transparent.sty' is not loaded}%
    \renewcommand\transparent[1]{}%
  }%
  \providecommand\rotatebox[2]{#2}%
  \ifx\svgwidth\undefined%
    \setlength{\unitlength}{34.379488bp}%
    \ifx\svgscale\undefined%
      \relax%
    \else%
      \setlength{\unitlength}{\unitlength * \real{\svgscale}}%
    \fi%
  \else%
    \setlength{\unitlength}{\svgwidth}%
  \fi%
  \global\let\svgwidth\undefined%
  \global\let\svgscale\undefined%
  \makeatother%
  \begin{picture}(1,0.99041615)%
    \put(0,0){\includegraphics[width=\unitlength]{crossing4.pdf}}%
  \end{picture}%
\endgroup%
} if $b_m<0$.
The odd case is proved in a similar way.
\end{proof}

\begin{cor} \label{cor 2}
 The signs of the crossings in the $b_i$-braid are $(-1)^{i+1}\sgn(b_i)$. 
 In particular,
 \begin{itemize}
\item [{\rm (a)}] the sign sequence $\sgn\cfb$ of the even continued fraction is equal to the  sequence of the signs of the crossings in the link diagram $C\cfb$.

\item [{\rm (b)}] the type sequence $\type\cfb$ of the even continued fraction is equal to the sequence of the  signs of the braids in $C\cfb$.  
\end{itemize}

\end{cor}
\begin{proof}
 This follows directly from Lemma \ref{lem 3}.
\end{proof}
\section{Jones polynomial}\label{sect Jones}
The Jones polynomial of an oriented link is an important invariant. For 2-bridge links, the Jones polynomial has been computed in \cite{LM,N}. 
 For general facts about the Jones polynomial see for example the book by Lickorish \cite{L}. 
 
 The Jones polynomial $V(L)$ of an oriented link $L$ can be defined recursively 
as follows. The Jones polynomial of the unknot is $1$, and whenever three oriented links $L_-,L_+$ and $L_0$ are the same except in the neighborhood of a point where they are as shown in Figure \ref{crossing5} then 
\[V(L_-) = t^{-2} \,V(L_+)+\ze\, V(L_0),\]
where $\ze= t^{-1}(t^{-1/2}-t^{1/2})= t^{-3/2}-t^{-1/2}$.
\begin{figure}
\begin{center}
  \scalebox{0.8}{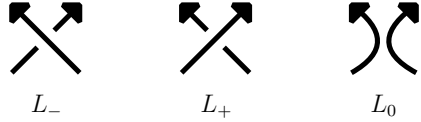}
 \caption{Definition of the Jones polynomial}
 \label{crossing5}
\end{center}
\end{figure}
 Equivalently,
\[V(L_+) = t^{2}\, V(L_+)+\bar\ze \,V(L_0),\]
where $\bar\ze= t^2(-\ze)=t(t^{1/2}-t^{-1/2})$.

\begin{remark}
 Usually the defining identity is stated as
 \[t^{-1} V(L_+)-t\, V(L_-)  + (t^{-1/2}-t^{1/2}) V(L_0)=0.\]

\end{remark}
\begin{example}The Jones polynomial  $V_{\!{\displaystyle\circ}\,{\displaystyle\circ}}$ of two disjoint copies of the unknot can be computed using the defining relation as shown on the top left of Figure \ref{fig Jonex}. Since $L_-$ and $L_+$ are both the unknot we see that \[V_{\!{\displaystyle\circ}\,{\displaystyle\circ}}= (1-t^{-2})/\ze=-t^{-{1}/{2}}-t^{{1}/{2}}.\]

The Jones polynomial of the Hopf link is computed in the top right of Figure \ref{fig Jonex}. We have \[V_{\textup{Hopf link}}=t^{-2} (-t^{-{1}/{2}}-t^{{1}/{2}})+\ze=-t^{-5/2}-t^{-1/2}.\]

The Jones polynomial of the trefoil knot is computed in the bottom left of Figure \ref{fig Jonex}. We have \[V_{\textup{trefoil}} =t^{-2} +\ze V_{\textup{Hopf link}}=t^{-2} +( t^{-3/2}-t^{-1/2})(-t^{-5/2}-t^{-1/2})=-t^{-4}+t^{-3}+t^{-1}.\]
\end{example}
\begin{figure}
\begin{center}
 \tiny \scalebox{1.5}{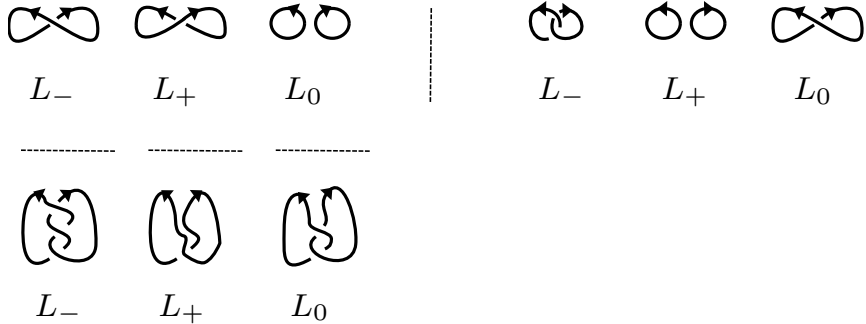}
 \caption{Computation of the Jones polynomial of two copies of the unknot ($L_0$ in the top left), the Hopf link ($L_-$ in the top right) and the  the trefoil knot  ($L_-$ in the bottom left)}
 \label{fig Jonex}
\end{center}
\end{figure}

The Jones polynomial is a Laurent polynomial in $t^{1/2}$. Let $\overline{(\ )}\colon \mathbb{Z}[t^{1/2},t^{-1/2}]\to \mathbb{Z}[t^{1/2},t^{-1/2}]$ be the algebra automorphism of order two that sends $t^{1/2} $ to $t^{-1/2}$. This is consistent with our definition of $\bar\ze$.
If $\overline L$ is the mirror image of the link $L$ then $V(\overline{L})=\overline{V(L)}$.

If $L$ is a 2-bridge link, $L=C\cfb=C(p/q)$ with $\cfb$ an even continued fraction, then the following are equivalent
\[ \textup{$L$ is a knot} \ \Longleftrightarrow\  \textup{$p$ is odd}\ \Longleftrightarrow\  \textup{$m$ is even}\ \Longleftrightarrow\  V(L)\in \mathbb{Z}[t,t^{-1}]. \]
If $L$ is not a knot then it is a 2-component link and $V(L)\in t^{1/2}\mathbb{Z}[t,t^{-1}] $.
If $L$ is a knot and $L'$ is  the same knot with reversed  orientation (running through $L $ in the opposite direction) then $V(L)=V(L')$. Thus the Jones polynomial of a knot does not depend on the orientation of the knot. If $L$ is a link with components $L_1,L_2$ and $L'$ is obtained from $L$ by reversing the orientation of one component, then 
\[V(L')=t^{-3\,lk(L_1,L_2)} V(L),\]
where $lk(L_1,L_2) = 1/2 (\textup{sum of signs of crossings between $L_1$ and $L_2$})$, see \cite[page 26]{L}.
\begin{example}\label{ex 4}
 Let $L=C[4]$ with first crossing \scalebox{.4}{
\begingroup%
  \makeatletter%
  \providecommand\color[2][]{%
    \errmessage{(Inkscape) Color is used for the text in Inkscape, but the package 'color.sty' is not loaded}%
    \renewcommand\color[2][]{}%
  }%
  \providecommand\transparent[1]{%
    \errmessage{(Inkscape) Transparency is used (non-zero) for the text in Inkscape, but the package 'transparent.sty' is not loaded}%
    \renewcommand\transparent[1]{}%
  }%
  \providecommand\rotatebox[2]{#2}%
  \ifx\svgwidth\undefined%
    \setlength{\unitlength}{34.4bp}%
    \ifx\svgscale\undefined%
      \relax%
    \else%
      \setlength{\unitlength}{\unitlength * \real{\svgscale}}%
    \fi%
  \else%
    \setlength{\unitlength}{\svgwidth}%
  \fi%
  \global\let\svgwidth\undefined%
  \global\let\svgscale\undefined%
  \makeatother%
  \begin{picture}(1,0.98982558)%
    \put(0,0){\includegraphics[width=\unitlength]{crossing2.pdf}}%
  \end{picture}%
\endgroup%
} and let $L'$ be the same link with first crossing \scalebox{.4}{
\begingroup%
  \makeatletter%
  \providecommand\color[2][]{%
    \errmessage{(Inkscape) Color is used for the text in Inkscape, but the package 'color.sty' is not loaded}%
    \renewcommand\color[2][]{}%
  }%
  \providecommand\transparent[1]{%
    \errmessage{(Inkscape) Transparency is used (non-zero) for the text in Inkscape, but the package 'transparent.sty' is not loaded}%
    \renewcommand\transparent[1]{}%
  }%
  \providecommand\rotatebox[2]{#2}%
  \ifx\svgwidth\undefined%
    \setlength{\unitlength}{34.05bp}%
    \ifx\svgscale\undefined%
      \relax%
    \else%
      \setlength{\unitlength}{\unitlength * \real{\svgscale}}%
    \fi%
  \else%
    \setlength{\unitlength}{\svgwidth}%
  \fi%
  \global\let\svgwidth\undefined%
  \global\let\svgscale\undefined%
  \makeatother%
  \begin{picture}(1,1.010279)%
    \put(0,0){\includegraphics[width=\unitlength]{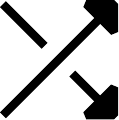}}%
  \end{picture}%
\endgroup%
}. Then $V(L)= -t^{1/2}+t^{3/2}-t^{5/2}-t^{9/2}$
 and
 $V(L')=
 -t^{-11/2}+t^{-9/2}-t^{-7/2}-t^{-3/2}$. On the other hand, the number $lk(L_1,L_2)$ is $\frac{1}{2}(-4)=-2$, and thus $V(L')= t^{-6}\,V(L).$
\end{example}

It is important to keep this small subtlety in mind when working with data bases, especially if it is not immediately obvious which conventions were used for the orientations of the link components. In this paper, the orientation is fixed.

\subsection{The Jones polynomial of a 2-bridge knot}
In this subsection, we compute recursion formulas for the Jones polynomials of 2-bridge knots and links. We also give a direct formula for the degree and compare the signs of the leading coefficients. Recall that because of our conventions (see section \ref{sect knots}), the orientation of the links is fixed and therefore the Jones polynomials are well defined.


 In what follows, we will work with recursive formulas for the Jones polynomials, and we will use the conventions that the expressions $C[b_1,b_2,\ldots,b_{0}]$ and $C[\ ]$ both denote the unknot and hence $V_{[b_1,b_2,\ldots,b_{0}]}=V_{[\ ]}=1$, and the expressions  $C[b_1,b_2,\ldots,b_{-1}]$ and $C[0]$ both denote  two disjoint copies of the unknot and hence $V_{[b_1,b_2,\ldots,b_{-1}]}=V_{[0]}=-t^{-1/2}-t^{1/2}$.
Recall the definition of the type sequence $\type\cfb$ in Definition~\ref{defsignb}. We will use the notation $\type\cfb=(\ldots,-)$ to indicate that the last entry of the type sequence is a minus sign.
\begin{lem}
\label{lem 4} Let $\cfa=p/q$ be a positive continued fraction and $\cfb=p/q$ or $p/(p-q)$ an even continued fraction and suppose that $m\ge 1$. So the two 2-bridge links $C\cfa$ and $C\cfb$ are isotopic. Let $V_{\cfa}$ or $V_{\cfb}$ denote their Jones polynomials. Thus $V_{\cfa}=V_{\cfb}$. Then

\[V_{\cfb}=\left\{\begin{array}{ll}
 t^{-2}\,V_{[b_1,b_2,\ldots, \,\sgn(b_{m})(|b_{m}|-2)\,]} +\ze\,V_{[b_1,b_2,\ldots,b_{m-1}]} 
 &\textup{if $\type[b_1,\ldots,b_{m}]=(\ldots,-)$};\\[8pt]
 t^{\,2}\ V_{[b_1,b_2,\ldots,{ \,\sgn(b_{m})(|b_{m}|-2)\,}]} +\overline{\ze}\,V_{[b_1,b_2,\ldots,b_{m-1}]}&\textup{if $\type[b_1,\ldots,b_{m}]=(\ldots,+)$.} 
\end{array}\right.
 \]
 In particular,
 \[V_{[b_1]}=\left\{\begin{array}{ll} 
 t^{-2}\,V_{[b_1+2]}+\ze&\textup{ if $b_1<0$;}\\[8pt]
t^{\,2}\ V_{[b_1-2]}+\overline\ze&\textup{ if $b_1>0$.}
\end{array}\right.\]
%
\end{lem}

\begin{proof} In type $(\ldots,-)$,
the sign of the last crossing in $C\cfb$ is negative. Therefore the relation $V(L_-)= t^{-2}\,V(L_+) +\ze\,V(L_0)$ applied to the last crossing yields the identity
 \[V_{\cfb} = t^{-2} \, V_{[b_1,b_2,\ldots,{ \,\textup{sign}(b_{m})(|b_{m}|-2)\,}]} + \ze\,V_
{[b_1,b_2,\ldots,b_{m-1}]}.\]
Indeed, by Lemma \ref{lem 3} the last crossing is \scalebox{0.4}{
\begingroup%
  \makeatletter%
  \providecommand\color[2][]{%
    \errmessage{(Inkscape) Color is used for the text in Inkscape, but the package 'color.sty' is not loaded}%
    \renewcommand\color[2][]{}%
  }%
  \providecommand\transparent[1]{%
    \errmessage{(Inkscape) Transparency is used (non-zero) for the text in Inkscape, but the package 'transparent.sty' is not loaded}%
    \renewcommand\transparent[1]{}%
  }%
  \providecommand\rotatebox[2]{#2}%
  \ifx\svgwidth\undefined%
    \setlength{\unitlength}{34.4bp}%
    \ifx\svgscale\undefined%
      \relax%
    \else%
      \setlength{\unitlength}{\unitlength * \real{\svgscale}}%
    \fi%
  \else%
    \setlength{\unitlength}{\svgwidth}%
  \fi%
  \global\let\svgwidth\undefined%
  \global\let\svgscale\undefined%
  \makeatother%
  \begin{picture}(1,0.98982558)%
    \put(0,0){\includegraphics[width=\unitlength]{crossing2.pdf}}%
  \end{picture}%
\endgroup%
} or  \scalebox{0.4}{
\begingroup%
  \makeatletter%
  \providecommand\color[2][]{%
    \errmessage{(Inkscape) Color is used for the text in Inkscape, but the package 'color.sty' is not loaded}%
    \renewcommand\color[2][]{}%
  }%
  \providecommand\transparent[1]{%
    \errmessage{(Inkscape) Transparency is used (non-zero) for the text in Inkscape, but the package 'transparent.sty' is not loaded}%
    \renewcommand\transparent[1]{}%
  }%
  \providecommand\rotatebox[2]{#2}%
  \ifx\svgwidth\undefined%
    \setlength{\unitlength}{34.4bp}%
    \ifx\svgscale\undefined%
      \relax%
    \else%
      \setlength{\unitlength}{\unitlength * \real{\svgscale}}%
    \fi%
  \else%
    \setlength{\unitlength}{\svgwidth}%
  \fi%
  \global\let\svgwidth\undefined%
  \global\let\svgscale\undefined%
  \makeatother%
  \begin{picture}(1,0.98982558)%
    \put(0,0){\includegraphics[width=\unitlength]{crossingaodd.pdf}}%
  \end{picture}%
\endgroup%
}. Therefore after the smoothing of this crossing, which leads to $L_0$, the last braid  becomes trivial, see Figure \ref{smoothing}.  
\begin{figure}
\begin{center}
 \huge \scalebox{0.5}{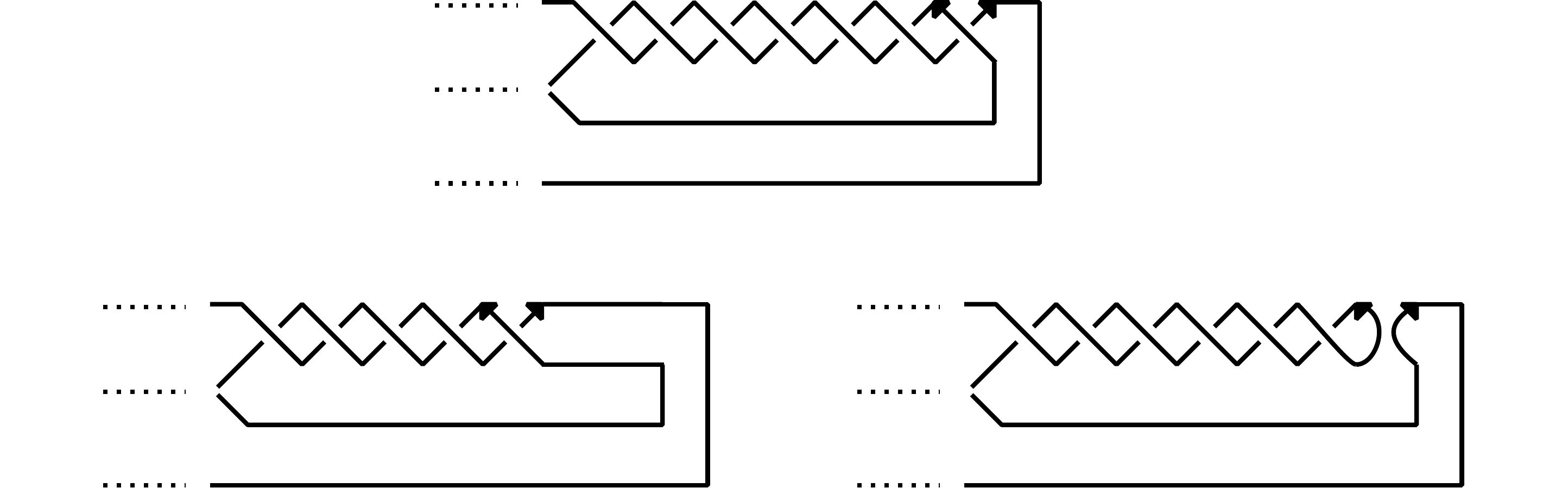}
 \caption{Proof of Lemma \ref{lem 4}}
 \label{smoothing}
\end{center}
\end{figure}

In type $(\ldots,+)$, the last crossing has positive sign and we use the analogous argument with the smoothing relation 
$V(L_+)= t^{2}\,V(L_-) +\overline\ze\,V(L_0)$.
\end{proof}

Applying the lemma $\frac{1}{2}|b_m|$ times, we obtain the recursive formula of the following theorem. In its statement, we will use the notation $[b]_q=1+q+q^2+\cdots +q^{b-1}=(1-q^b)/(1-q)$ for the $q$-analogue of a positive integer $b$. Setting $q=-t^{-1}$ and $\overline{q}=-t$, this gives the short hands
\[[b]_q=1-t^{-1}+t^{-2} -t^{-3}+\cdots \pm t^{-b+1} \qquad\textup{ and }\qquad [b]_{\overline{q}}=1-t^{1}+t^{2} -t^{3}+\cdots \pm t^{b-1}.\]
\begin{thm}
 \label{thm 5} Let $\cfb$ be an even continued fraction with $m\ge1$. Then
\[V_{\cfb}= 
 \left\{
\begin{array}
 {ll}
 t^{-|b_m|}\,V_{[b_1,b_2,\ldots,b_{m-2}]} -t^{-\frac{1}{2}}[\,|b_m|\,]_q\,V_{[b_1,b_2,\ldots,b_{m-1}]} 
&\textup{if $\type[b_1,\ldots,b_{m}]=(\ldots,-)$:}\\[8pt]
t^{|b_m|}\ V_{[b_1,b_2,\ldots,b_{m-2}]} -\ t^{\frac{1}{2}}\ [\,|b_m|\,]_{\overline {q}}\,V_{[b_1,b_2,\ldots,b_{m-1}]} 
&\textup{if $\type[b_1,\ldots,b_{m}]=(\ldots,+)$.}\\
\end{array}\right.
 \]
  In particular, 
  \[V_{[b_1]}= 
 \left\{
\begin{array}
 {ll}
 t^{b_1}(-t^{-1/2}-t^{1/2})-t^{-\frac{1}{2}}\,[-b_1]_q, 
  &\textup{if $b_1<0$;}\\
  t^{b_1}(-t^{-1/2}-t^{1/2})-t^{\frac{1}{2}}\,[b_1]_{\overline q}
  &\textup{if $b_1>0$.}\\
\end{array}\right.
\]
\end{thm}
%
%

\begin{proof}
 This follows simply by applying Lemma \ref{lem 4} exactly $b_m/2$ times and using the facts that 
 \[\ze(1+t^{-2}+t^{-4}+\cdots+t^{-|b_m| +2}) = -t^{-\frac{1}{2}}[\,|b_m|\,]_q ,
 \] 
 \[
 \overline{\ze}(1+t^{2}+t^{4}+\cdots+t^{|b_m| -2}) = -t^{\frac{1}{2}}[\,|b_m|\,]_{\overline{q}},\] and
 $C[b_1,b_2,\ldots,b_{m-1},0]=C[b_1,b_2,\ldots,b_{m-2}]$.  

 For  the case where $m=1$, we also need to observe that $C{[0]}$ is a disjoint union of two unknots, and thus $V_{[0]}=-t^{-1/2}-t^{1/2}$, and $C[\ ]$ is the unknot, hence $V_{[\ ]}=1$
\end{proof}

In the remainder of the section, we compute the degree and the sign of the leading term of the Jones polynomial.  The degree of a Laurent polynomial in $t^{1/2}$ is the highest exponent of $t$ that appears, and the leading term is the term that realizes the degree. For example, the degree of $t^{-3/2}-t^{-1/2}$ is $-1/2$ and the leading term is $-t^{-1/2}$.

For simplicity, we shall use the following notation for $i=0,1,2,\ldots, m-1$.
\[\begin{array}
 {cccccc} 
 V_i=V_{[b_1,b_2,\ldots,b_{m-i}]} & \quad & j_i= \deg V_i& \quad & \zd_i t^{j_i}= \textup{leading term of } V_i. 
\end{array}
\]

\begin{cor} 
 \label{cor 6} With the notation above
 \[j_0 \, \le \left\{
\begin{array}
 {ll} 
 \max (j_2 -|b_m| \ , \ j_1 -\textstyle\frac{1}{2} ) &\textup{if $\type[b_1,\ldots,b_{m}]=(\ldots,-)$; }\\ \\
\max (j_2 +|b_m | \ , \ j_1 -\textstyle\frac{1}{2}+|b_m|)
&\textup{if $\type[b_1,\ldots,b_{m}]=(\ldots,+)$.}
\end{array}\right. \]
 with equality if $j_2 -|b_m| \ne  j_1 -\textstyle\frac{1}{2}$, respectively $j_2 +|b_m | \ne j_1 -\textstyle\frac{1}{2}+|b_m|$.
\end{cor}
\begin{proof}
 This follows directly from Theorem \ref{thm 5} and the fact that $\deg[b]_q =1$ and $\deg [b]_{\overline {q}}=b-1$, for $b>0$.
 \end{proof}

\begin{cor}
 \label{cor 7}
The degrees $j_0,j_1 $ and $j_2$ compare as follows accoding to the type of $\cfb$.
 \[
\begin{array}
 {rcl} j_0 &=&j_1+ \left\{
\begin{array}
 {ll} 
 -\frac{1}{2} &\textup{in type }(\ldots,-);\\[8pt]
  |b_m|+\frac{1}{2}\qquad&\textup{in type }(\ldots,-,+);\\[8pt]
 |b_m|-\frac{1}{2}&\textup{in type }(\ldots,+,+);\\[8pt]
\end{array}
\right.\\
\\
j_0&=&j_2+\left\{
\begin{array}
 {ll} 
  -1 &\textup{in type }(\ldots,-,-);\\[8pt]
   |b_{m-1}|   &\textup{in type }(\ldots,-,+,-);\\[8pt] 
   |b_{m-1}| - 1 &\textup{in type }(\ldots,+,+,-);\\[8pt]
   |b_{m}|   &\textup{in type }(\ldots,-,+);\\[8pt] 
   |b_{m}|  + |b_{m-1}|  &\textup{in type }(\dots,-,+,+);\\[8pt] 
   |b_{m}|  + |b_{m-1}| -1 &\textup{in type }(\dots,+,+,+).\\[8pt] 
\end{array}
\right.
\end{array}\]
%
%
\end{cor}

\begin{proof}
We use induction on $m$.  If $m=1$, we have  $j_1=\deg V_{\textup{unknot}}=0$ and $j_2=\deg (-t^{-1/2}-t^{1/2})=1/2$.  On the other hand, Theorem \ref{thm 5}  implies
\[j_0=\left\{
\begin{array}
 {ll}
 -\frac{1}{2} &\textup{if $b_1<0$};\\
 b_1+\frac{1}{2} &\textup{if $b_1>0$}.
\end{array}\right.\]
Therefore 
\[j_0=\left\{
\begin{array}
 {ll}
 j_1-\frac{1}{2} &\textup{if $b_1<0$};\\
 j_1+b_1+\frac{1}{2} &\textup{if $b_1>0$},
\end{array}\right.
\qquad \textup{and} \qquad 
j_0=\left\{
\begin{array}
 {ll}
 j_2-1 &\textup{if $b_1<0$};\\
 j_2+b_1 &\textup{if $b_1>0$}.
\end{array}\right.
\]
This shows the result for $m=1$.

Suppose now $m>1$. By induction, we have 
\begin{equation}
 \label{eq 58}
 j_1=j_2+\left\{
\begin{array}
 {ll}
 -\frac{1}{2} &\textup{if $\type[b_1,\ldots,b_{m-1}]=(\ldots,-)$};\\[5pt]
 |b_{m-1}|+\frac{1}{2} &\textup{if $\type[b_1,\ldots,b_{m-1}]=(\ldots,-,+)$};\\[5pt]
 |b_{m-1}|-\frac{1}{2} &\textup{if $\type[b_1,\ldots,b_{m-1}]=(\ldots,+,+)$}.\\
 \end{array}\right.
\end{equation}
In particular, $j_1\ge j_2- 1/2$, whence $j_1-1/2>j_2-|b_m|$, since $b_m$ is a nonzero even integer. Therefore Corollary \ref{cor 6} implies that, if $\type\cfb=(\ldots,-)$, then $j_0=j_1-1/2,$ and the three cases in equation (\ref{eq 58}) prove the three cases of statement where the type ends in a minus sign. 

Next suppose that $\type\cfb=(\ldots,+,+)$. In this case, equation (\ref{eq 58}) implies that  $j_1=j_2+|b_{m-1}| \pm 1/2>j_2+1$. Then Corollary \ref{cor 6} yields
$j_0=j_1-1/2+|b_m|=j_2+|b_{m}|+|b_{m-1}]-\frac{1}{2}\pm\frac{1}{2}.$

Finally, suppose that $\type\cfb=(\ldots,-,+)$.  Then equation (\ref{eq 58}) implies that  $j_1=j_2-1/2$ and thus $j_1-1/2<j_2$. In this case, Corollary \ref{cor 6}   yields $j_0=j_2+|b_m|=j_1+|b_m|+1/2$.
%
\end{proof}

Corollary \ref{cor 7} allows us to determine which of the two polynomials on the right hand side of the equations in Theorem \ref{thm 5} contains the leading term. Indeed, if $\cfb$ is of type $(\ldots,-)$, the equation is 
\[V_0= 
 t^{-|b_m|}\,V_2 -t^{-\frac{1}{2}}[\,|b_m|\,]_q\,V_1. 
\]
The leading term of the first polynomial on the right hand side has sign $\zd_2$ and its degree $j_2-|b_m|$ is strictly smaller than $j_2-1\le j_0$, by Corollary \ref{cor 7}. The leading term of the second polynomial has sign $-\zd_1$ and its degree is 
 $j_1-\frac{1}{2}=j_0$.
Therefore the leading term of $V_0$ is $\zd \,t^{j_0}=-\zd_1\,t^{j_1-\frac{1}{2}}$.
 
 On the other hand, if $\cfb$ is of type $(\ldots,+)$, the equation in Theorem \ref{thm 5} is 
\[V_0= 
 t^{|b_m|}\,V_2 -t^{\frac{1}{2}}[\,|b_m|\,]_{\overline q}\,V_1. 
\]
The first term on the right hand side has sign $\zd_2$ and its degree $j_2+|b_m|$ is equal to $j_0$ in type $(\ldots,-,+)$. The second term has sign $\zd_1$ (since $b_m$ is even) and its degree is 
$j_1+ |b_m| -\frac{1}{2}$ which is equal to $j_0$ in type $(\ldots,+,+)$.
This leads us to the following corollary.

\begin{cor}\label{cor 8}
The leading terms satisfy
\begin{equation}\label{eq 59} \zd_0\,t^{j_0}=\left\{\begin{array}{ll}
-\zd_1\,t^{j_1-\frac{1}{2}} &\textup{ in type $(\ldots,-)$;}\\[5pt]
\zd_2\,t^{j_2+|b_m|} &\textup{ in type $(\ldots,-,+)$;}\\[5pt]
\zd_1\,t^{j_1+|b_m|-\frac{1}{2} } &\textup{ in type $(\ldots,+,+)$.}\end{array}\right.\end{equation}
Moreover, the coefficients $\zd_0,\zd_1$ and $\zd_2$ compare according to the type of $\cfb$ as follows.
\[\zd_0=\left\{\begin{array}{ll}
-\zd_1 &\textup{ in type $(\ldots,-)$;}\\
-\zd_1 &\textup{ in type $(\ldots,-,+)$;}\\
\zd_1  &\textup{ in type $(\ldots,+,+)$;}\\
\end{array}\right.\qquad \textup{and}\qquad 
\zd_0=\left\{\begin{array}{ll}
\zd_2 &\textup{ in type $(\ldots,-,-)$;}\\
\zd_2  &\textup{ in type $(\ldots,-,+,-)$;}\\
-\zd_2 &\textup{ in type $(\ldots,+,+,-)$;}\\
\zd_2 &\textup{ in type $(\ldots,-,+)$;}\\
-\zd_2  &\textup{ in type $(\ldots,-,+,+)$;}\\
\zd_2 &\textup{ in type $(\ldots,+,+,+)$.}\\
\end{array}\right.\]
\end{cor}
 
\begin{proof}
 It only remains to show the last equations comparing $\zd_0$ with $\zd_1$ and $\zd_2$.
 Two of the three equations comparing $\zd_0 $ and $\zd_1$ follow directly from
equation (\ref{eq 59}). In the remaining equation, the type is $(\ldots,-,+)$ and equation (\ref{eq 59}) yields $\zd_0=\zd_2$ as well as $\zd_1=-\zd_2$, because $\type[b_1,\ldots,b_{m-1}]=(\ldots,-)$. Thus $\zd_0=-\zd_1$.

Now consider the equations comparing $\zd_0$ and $\zd_2.$ In type $(\ldots,-,-)$ we have $\zd_0=-\zd_1$ and $\zd_1=-\zd_2$, thus $\zd_0=\zd_2$.
In type $(\ldots,+,-)$ we have $\zd_0=-\zd_1$, wheras  $\zd_1=-\zd_2$ in type $(\ldots,-,+,-) $ and  $\zd_1=\zd_2$ in type $(\ldots,+,+,-) $. Thus $\zd_0=\zd_2$ in the former case and $\zd_0=-\zd_2$ in the latter.
 
 In type $(\ldots,-,+)$ we have $\zd_0=-\zd_1$ and $\zd_1=-\zd_2$, thus $\zd_0=\zd_2$.
 In type $(\ldots,+,+)$ we have $\zd_0=\zd_1$, wheras  $\zd_1=-\zd_2$ in type $(\ldots,-,+,+) $ and  $\zd_1=\zd_2$ in type $(\ldots,+,+,+) $. Thus $\zd_0=-\zd_2$ in the former case and $\zd_0=\zd_2$ in the latter.
This completes the proof.
\end{proof}

 We close this section with a direct formula for the degree of the Jones polynomial.

\begin{thm}
 \label{thm degree}
 Let $\cfb$ be an even continued fraction. Then the Jones polynomial $V_{\cfb}$ of the associated 2-bridge knot has degree
 \[ \sum_{i=1}^m  \max\left( (-1)^{i+1} b_i + \frac{\textup{sign}(b_ib_{i-1})}{2}\ ,\ -\frac{1}{2}\right),\]
 where we use the convention that $\textup{sign}(b_0)=1$. 
  Moreover, the sign of its leading term is equal to $(-1)^{m-\tau}$, where $\tau$ is the number of times the subsequence $+,+$ appears in the type sequence of $\cfb$.
\end{thm}
\begin{proof} If $m=1$, then Theorem \ref{thm 5} implies  that the leading term of $V_{b_1}$ is
\[
\left\{
\begin{array}
 {ll}-t^{ b_1+\frac{1}{2}} &\textup{if $b_1>0$;}\\
 -t^{-\frac{1}{2}} &\textup{if $b_1<0$,}
\end{array}\right. \]
so the degree is equal to
$ \max( b_1 + \frac{\textup{sign}(b_1)}{2}\ ,\ -\frac{1}{2})$ and the sign is equal to $(-1)=(-1)^m$.

Now suppose that $m>1$. Corollary \ref{cor 7} implies
\[\deg V_{\cfb}=j_0=j_1 +\left\{
\begin{array}
 {ll}
 -\frac{1}{2}  &\textup{in type $(\ldots,-)$;}\\
 |b_m| +\frac{\textup{sign}(b_m b_{m-1})}{2} &\textup{in type $(\ldots,+)$.}\\
\end{array}\right. \]
By induction, we may assume that $j_1 $ is equal to the sum of the first $m-1$ terms in the theorem. Thus we must show that the $m$-th term satisfies
 \[   \max\left( (-1)^{m+1} b_m + \frac{\textup{sign}(b_mb_{m-1})}{2}\ ,\ -\frac{1}{2}\right) = \left\{
\begin{array}
 {ll}
 -\frac{1}{2}  &\textup{in type $(\ldots,-)$;}\\
  |b_m|+\frac{\textup{sign}(b_m b_{m-1})}{2}  &\textup{in type $(\ldots,+)$.}\\
\end{array}\right. \]
Recall that $\type\cfb=(\ldots,-)$  if  $m$ is even and $b_m>0$, or if $m$ is odd and $b_m<0$. Therefore the maximum on the left hand side is equal to $-1/2$ in this case. 
On the other hand, if  $\type\cfb=(\ldots,+)$,   then the maximum equals $|b_m|
 +\frac{\textup{sign}(b_m b_{m-1})}{2} $.  This proves the statement about the degree. To determine the sign, we use Corollary \ref{cor 8}, which shows how the sign changes in terms of the entries at positions $2,3,\ldots,m$ in the type sequence of $\cfb$. Namely,  the sign changes for every $-$ sign in these positions of the type sequence  and for each $+$ sign that is a direct successor of a $-$ sign. In other words, the number of sign changes is precisely $m-1-\tau$. Now the result follows since the sign is $-$ in the case $m=1$. 
\end{proof}
\begin{remark}
 For each $i$, the maximum in the theorem is equal to $-1/2$ if the crossings  in the $b_i$ braid are negative and it is $|b_i|\pm 1/2$ if the crossings are positive. Thus we can express the degree in terms of the crossings of the link as follows
 \[\begin{array}{rcl}\deg V_{\cfb}&=&-\frac{1}{2}\#\,\textup{negative braids\,} +\#\,\textup{positive crossings}
 \\[8pt]
 &&+\frac{1}{2} \#\,\textup{consecutive pairs of braids with sign } (+,-)\\[8pt]
 &&-\frac{1}{2}\#\,\textup{consecutive pairs of braids with sign }(+,+).  \end{array}\]
\end{remark}
\begin{example} For the continued fraction $[2,2,-2,4]$, the formula of the theorem gives
 \[\deg V_{[2,2,-2,4]}=\left(2+\frac{1}{2}\right)-\frac{1}{2}-\frac{1}{2}-\frac{1}{2}=1,\]
 and the formula from the remark gives
 \[\deg V_{[2,2,-2,4]}=-\frac{3}{2}+2+\frac{1}{2}=1.\]
  The type sequence of $[2,2,-2,4]$ is $(+,-,-,-)$, hence $\tau=0$ and thus the sign of the 
 leading term is $(-1)^4=+1$.
  
The Jones polynomial is 
\[V_{[2,2,-2,4]}=t-2+4 t^{-1} -4 t^{-2} +5 t^{-3} -5 t^{-4} +3 t^{-5} -2 t^{-6} + t^{-7} .\]
 
\end{example}

\section{Cluster algebras and specialized $F$-polynomials}\label{sect F}
\subsection{Cluster algebras}
We recall a few facts about cluster algebras with principal coefficients. For more detailed information, we refer to the original paper \cite{FZ4} or the lecture notes \cite{S}.
Let $Q$ be a quiver without loops and oriented 2-cycles, and let $N$ denote the number of vertices of $Q$. Let $\mathbb{ZP}$ denote the ring of Laurent polynomials in variables $y_1,y_2,\ldots,y_N$ and let $\mathbb{QP}$ be its field of fractions. The {\em cluster algebra} $\cala(Q)$ of $Q$ with {\em principal coefficients}  is a $\mathbb{ZP}$-subalgebra  of the field of rational functions $\mathbb{QP}(x_1,x_2,\ldots,x_N)$. To define the cluster algebra one constructs a set of generators, the {\em cluster variables}, by a recursive method called {\em mutation}. It is known  that the cluster variables are elements of the ring $\mathbb{Z}[x_1^{\pm 1},x_2^{\pm 1},\ldots,x_N^{\pm 1},y_1,y_2,\ldots,y_N]$  with positive coefficients \cite{FZ1,FZ4,LS4}. The {\em $F$-polynomial} is the polynomial in $\mathbb{Z}[y_1,y_2,\ldots,y_N]$ obtained from the cluster variable by setting all $x_i$ equal to 1. 

If the quiver $Q$ is the adjacency quiver of a triangulation of a surface with marked points then the cluster algebra is said to be of surface type, see \cite{FST}. In this case, each cluster variable is given as a sum over all perfect matchings of the weighted snake graph associated to the cluster variable \cite{MSW}. In \cite{CS4,R} another formula was given, that computes the cluster variables as continued fractions of Laurent polynomials.

\begin{example}
 Let $Q$ be the quiver $\xymatrix{1&\ar[l]2}$. Then the `largest' cluster variable in $\cala(Q)$ is equal to $(x_2+y_1+x_1y_1y_2)/x_1x_2$ 
 and its $F$-polynomial is $1+y_1+y_1y_2$.
\end{example}

\begin{example}
 Let $Q$ be the quiver $\xymatrix{1&\ar[l]2\ar[r]&3}$. Then the `largest' cluster variable in $\cala(Q)$ is equal to $(x_2^2+x_2y_1+x_2y_3+y_1y_3+x_1x_3y_1y_2y_3)/x_1x_2x_3$ 
 and its $F$-polynomial is $1+y_1+y_3+y_1y_3+y_1y_2y_3$.
\end{example}

\subsection{Specialized $F$-polynomials} We shall show that the Jones polynomial of a 2-bridge link is equal  (up to normalization by its leading term) to the specialization of a corresponding cluster variable at $x_i=1, y_1=t^{-2}$ and $y_i=-t^{-1} $ if $i\ne 1$.

To make this statement precise, we need to fix our notation. Let $\cfa$ be a positive continued fraction with $a_1\ge 2$. As we have seen in section \ref{sect knots}, this is not a restriction from the point of view of 2-bridge knots. 
Let $\calg\cfa$ be the snake graph of the continued fraction, let $d=a_1+a_2+\cdots+a_n-1$ be the number of tiles of this graph and label the tiles $1,2,\ldots,d$. 

Let $\cala$ be any cluster algebra with principal coefficients in which we can realize the snake graph $\calg\cfa$  as the snake graph of a cluster variable, with the sole condition that the first tile of  the snake graph corresponds to the initial cluster coefficient $y_1$ and no other tile of the graph corresponds to the same coefficient $y_1$. 
Denote the initial seed (with principal coefficients) by 
$((x_1,x_2,\ldots,x_N),(y_1,y_2,\ldots,y_N),Q)$ 
where $(x_1,x_2,\ldots,x_N)$ is the initial cluster, $(y_1,y_2,\ldots,y_N) $ is the initial coefficient tuple and $Q$ is the initial quiver.

\begin{remark}
  For example, one can choose $\cala$ to be of type $\mathbb{A}_d$ and let the $i$-th tile  correspond to the $i$-th coefficient and choose the initial seed such that its quiver $Q$ is the acyclic quiver corresponding to the snake graph $\calg\cfa$. That means that 
$Q$ is of the form 
\[\xymatrix@C15pt{1&\ar[l]2&\ar[l]\cdots&\ar[l]\ell_1\ar[r]&\ell_1+1\ar[r]&\cdots\ar[r]&\ell_2&\ell_2+1\ar[l]&\ar[l]\cdots&\ar[l]\ell_3\ar[r]&\cdots}
\]
where $\ell_i=a_1+a_2+\cdots a_i$.
In this cluster algebra, the cluster variable corresponding to the continued fraction $\cfa$ is the one whose denominator is the product of all initial cluster variables $x_1x_2\ldots x_d$, or equivalently, the cluster variable that,  under the Caldero-Chapoton map, corresponds to the largest indecomposable representation of $Q$, the one with dimension 1 at every vertex.

\end{remark}

In the chosen cluster algebra $\cala$, let $x\cfa$ denote the cluster
variable whose snake graph is $\calg\cfa$ and let $F\cfa$ be its 
$F$-polynomial. Recall from \cite{FZ4} that  $F\cfa$ is obtained from $x\cfa$ by setting all initial cluster variables $x_1,x_2,\ldots x_N$ equal to 1.

It has been shown in \cite{CS4, R} that $x\cfa $ and $F\cfa$ can be written as the numerator of a continued fraction of Laurent polynomials.
We will use a specialization of the $F$-polynomial by setting 
$y_1=q^2$ and $y_i=q$, for all $i=2,3,\ldots,N$, where $q=-t^{-1}$. In other words
 \[y_1=t^{-2}  \qquad\textup{ and } \qquad y_2=y_3=\ldots=y_N=-t^{-1}.\]
 We denote this specialization by $\Fa$. Thus
 \[\Fa=F\cfa \Big|_{y_1=t^{-2};\  y_i=-t^{-1},\ i>1}.\]
We do not yet know an intrinsic reason why $y_1$ is different from $y_2,\ldots,y_N$.

Using the main result of \cite{R}, we have the following formula.
Recall that  $[b]_q=1+q+q^2+\cdots +q^{b-1}$ and $\ell_i=a_1+a_2+\cdots a_i$. Also recall that  the notation $N[\call_1,\call_2,\ldots,\call_n]$ is defined recursively in equation (\ref{def N}), where $N[\call_1]=\call_1$.
\begin{prop}
 \label{prop F} 
\begin{itemize} 
\item [{\rm(a)}] 
If $n$ is odd, 
the specialized $F$-polynomial $\Fa$ is equal to 
\[N\Big[\, [a_1+1]_q -q\ ,\  [a_2]_q \,q^{-\ell_2} ,\ [a_3]_q\,q^{\ell_2+1} ,\ [a_4]_q\, q^{-\ell_4},\ldots, [a_{2i}]_q\, q^{-\ell_{2i}},\ [a_{2i+1}]_q\, q^{\ell_{2i}+1} ,\ldots , \ [a_{n}]_q\, q^{\ell_{n-1}+1}\Big] \]

\item [{\rm(b)}] If $n$ is even, $\Fa$ is equal to the result in {\rm(a)} multiplied by  $q^{\ell_n}$.
\end{itemize}
\end{prop}

\begin{proof} In her formula  \cite{R},
 Rabideau uses the  following notation.
\[\begin{array}
{rclcrcl}	
C_i&=& \left\{ \begin{array}{ll} 
\displaystyle  \prod_{j=1}^{\ell_{i-1}} y_j & \textup{if $ i $ is odd},\\ \\
\displaystyle \prod_{j=1}^{(\ell_i)-1}y_j^{-1} & \textup{if $i$ is even,}\
\end{array} \right.
&\qquad&			
\varphi_i&=& \left\{\begin{array}{ll}
\displaystyle \sum_{k=\ell_{i-1}}^{(\ell_i) -1}\  \prod_{j=(\ell_{i-1})+1}^{k} y_j  & \textup{if $i$ is odd},\\ \\
\displaystyle   \sum_{k=(\ell_{i-1})+1}^{\ell_i}\ \prod_{j=k}^{(\ell_i)-1} y_j & \textup{if $i$ is even,}
\end{array} \right.
\end{array}\]
and $\call_i=C_i\varphi_i.$
With this notation 
\[F\cfa=\left\{
\begin{array}
{ll}
N[\call_1,\call_2,\ldots,\call_n]  & \textup{if $n $ is odd},\\ \\
C_n^{-1}N[\call_1,\call_2,\ldots,\call_n]  & \textup{if $n $ is even}.
\end{array}\right.\]

Under our specialization $y_1\mapsto q^2$ and $y_i\mapsto q$ $(i>1)$, the quantities above transform as follows. 
\[\begin{array}
{rclcrcl}	
C_i &\mapsto&
\left\{ \begin{array}{ll} 
 q^{\ell_{i-1}+1}  & \textup{if $i $ is odd};\\ \\
 q^{-\ell_{i}}  & \textup{if $i$ is even;}
\end{array} \right.
\\ \\
\displaystyle\varphi_1=\sum_{k=0}^{a_1 -1}\  \prod_{j=1}^{k} y_j & \mapsto& 1+q^2+q^3+\cdots +q^{a_1}\ =\  [a_1+1]_q-q;\\
\varphi_i&\mapsto & 
\displaystyle \sum_{k=0}^{\ell_i-\ell_{i-1} -1}\   q^k \ =\  [a_i]_q & (i>1).
\end{array}\]
Thus
\[\call_1\mapsto  [a_1+1]_q-q \quad \textup { and } \quad \call_i \mapsto 
\left\{ \begin{array}{ll} 
  [a_i]_q \,q^{-\ell_i}  & \textup{if $i $ is odd};\\ \\ {}
[a_i]_q \,q^{\ell_{i-1}+1}  & \textup{if $i$ is even,}
\end{array} \right.	
\]
and the proof is complete.
\end{proof}

\begin{remark}
 If the $a_i$ are positive integers then $N\cfa$ is the numerator of the continued fraction, or to be more precise, the positve numerator, since numerators are only defined up to multiplication by units $\pm 1$ in $\mathbb{Z}$. 
 
 Since the $\call_i\in \mathbb{Z}[t,t^{-1}]$ are Laurent polynomials then $N[\call_1,\call_2,\ldots,\call_n]$ is also a representative of the numerator of the continued fraction. Note however that the powers of $t$ are also units in this ring. This is the reason why we define the notation $N$ by the recursion in equation (\ref{def N}).
\end{remark}
\subsection{$F$-polynomials of even continued fractions}
We now want to define  specialized $F$-polynomials for even continued fractions. If the value of the even continued fraction is equal to a positive rational number $r/s$, then its $F$-polynomial is the $F$-polynomial of the positive continued fraction of $r/s$. However, if the value of the even continued fraction is negative, our definition involves the bar automorphism defined in section \ref{sect knots}.  Note that $r/s>0$ if and only if $b_1>0$.

\begin{definition}
 \label{def F}
 Let $\cfb$ be an even continued fraction with value $r/s\in \mathbb{Q}$. Let $\cfa$ be the positive continued fraction expansion of the absolute value  $|r/s|$ of $r/s$ and let $d=a_1+\cdots +a_n-1$. We define the specialized $F$-polynomial of $\cfb$ as follows.
 \[ \Fb = \left\{\begin{array}{rl}
 \Fa &\textup{if $b_1>0$;}\\[8pt]
 (-t^{-1})^{d+1}\,\overline{\Fa} &\textup{if $b_1<0$.} \end{array}\right.\\
  \]
\end{definition}

\begin{example}\label{ex 6.4} The even continued fraction $[2,-2]$ is equal to $3/2$. Thus
 $F_{2,-2}=F_{1,2}$ which is equal    \[q^3N[\, [2]_q-q, [2]_q \,q^{-3}\,]=([2]_q-q) [2]_q + q^3 =(1)(1+q)+q^3= 1+q+q^3=1-t^{-1}-t^{-3}.\] Thus
 \[F_{2,-2}=1-t^{-1}-t^{-3}.\]
 On the other hand, $[-2,2]=-3/2$. Thus
  \[F_{-2,2}=(-t^{-1})^3\,\overline{F_{1,2}} =-t^{-3}(1-t^{1}-t^{3})=-t^{-3}+t^{-2}+1.\]
  Note that we also have $F_{-2,2}=F_{3}$.
\end{example}

\begin{example}\label{ex 6.5} The even continued fraction $[4]$ is equal to $4/1$. Thus
\[F_{4}= [5]_q-q =1+t^{-2}-t^{-3}+t^{-4}.\]
 
 On the other hand, 
  \[F_{-4}=(-t^{-1})^4\,\overline{F_{4}} =t^{-4}+t^{-2}-t^{-1}+1.\]
  Note that we also have $F_{-4}=F_{1,3}$.
\end{example}

\begin{example} \label{ex 6.6}
%
%
%
The even continued fraction $[4,-2]$ is equal to $7/2$ which is equal to the positive continued fraction $[3,2]$. Thus
\[ \begin{array}{rcl}F_{4,-2}&=&F_{3,2}= q^5N[ \,[4]_q -q\,,\,[2]_q\,q^{-5}\,]= ([4]_q-q)[2]_q+q^5\\
&=& 
1-t^{-1}+t^{-2}-2t^{-3}+t^{-4}-t^{-5}.\end{array}\]
On the other hand,
\[F_{-4,2}=(-t^{-1})^5\, \overline{F_{3,2}} =- t^{-5}+t^{-4}-t^{-3}+2t^{-2}-t^{-1}+1.\]
Note that $F_{-4,2}=F_{1,2,2}$.
\end{example}

\begin{lem} \label{lemF} Let $y_{i_j} $ denote the coefficient variable in the cluster algebra that corresponds to the $j$-th tile in the snake graph $\calg\cfa\cong\calg\cfb$, where $j=1,\ldots ,d$. 
 \begin{itemize}
\item [{\rm (a)}]  The $F$-polynomial $F\cfa$ is a polynomial in $\mathbb{Z}_{\ge 0} [y_1,y_2\ldots,y_N]$ of the form \[F\cfa=\sum \chi_{\za_1,\ldots,\za_d}\, y_{i_1}^{\za_1}\cdots y_{i_d}^{\za_d},\] where $\za_j\in\{0,1\}$, with constant term 1 and highest degree term $y_{i_1}y_{i_2}\cdots y_{i_d}$.

\item [{\rm (b)}] The specialized $F$-polynomial $\Fb$ is a polynomial in $\mathbb{Z}[t^{-1}]$ of the form \[\Fb=\sum_{i=0}^{d+1} (-1)^i \zs_i\, t^{-i},\] where $\zs_i\in \mathbb{Z}_{\ge 0}$, with constant term 1 and lowest degree term $(-1)^{d+1}\,t^{-d-1}$. In particular, the degree of $\Fb$ is 1.
\end{itemize}
\end{lem}
\begin{proof}
 Part (a) was conjectured in \cite{FZ4} and proved in \cite{DWZ2} except for positivity which is shown in \cite{MSW} for surface type and in \cite{LS4} for arbitrary quivers. The sum in (b) is alternating, since the sum in (a) is positive. The lowest degree in (b) is $-d-1$, since $y_1 $ is specialized to $t^{-2}$ and every other $y_i$ is specialized to $-t^{-1}$.
\end{proof}

The following proposition has already been observed in the examples above. 
In the proof, we shall work with the {\em minimal matching} $P_-$ of the snake graph $\calg\cfa$. The snake graph has precisely two perfect matchings $P_-$ and $P_+$ which contain only boundary edges of the snake graph. By convention, the minimal matching $P_-$ is the one that contains the edge $e_0$, the south edge of the first tile. The minimal matching has trivial coefficient $y(P_-)=1$.  The matching $P_+$ is called the maximal matching of the snake graph, and its coefficient  $y(P_+)=y_{i_1}y_{i_2}\cdots y_{i_d}$  is the product of the $y$-coefficients of all tiles in the snake graph.
The term $y(P_-)$ is the constant term in the  polynomial  $F\cfa$ and $y(P_+)$ is the term of highest degree.

\begin{prop}
 \label{prop Fbar} Let $\cfa$ be a positive continued fraction and let $d
 =a_1+a_2+\cdots+a_n-1$. Then
\[
\begin{array}{rcl} F\cfa&=& y_{i_1} y_{i_2}\cdots y_{i_d} \ \overline{F[1,a_1-1,a_2,\ldots,a_n]}.\\ \\
\Fa &=&q^{d+1} \ \overline{F_{1,a_1-1,a_2,\ldots,a_n}}.
\end{array}
\]
\end{prop}
\begin{proof}
 The snake graph $\calg[1,a_1-1,a_2,\ldots, a_n]$ is obtained from the snake graph $\calg\cfa$ by a reflection along the line containing the diagonal of the first tile. For example,  
 \[{\LARGE\scalebox{0.55}{
\begingroup%
  \makeatletter%
  \providecommand\color[2][]{%
    \errmessage{(Inkscape) Color is used for the text in Inkscape, but the package 'color.sty' is not loaded}%
    \renewcommand\color[2][]{}%
  }%
  \providecommand\transparent[1]{%
    \errmessage{(Inkscape) Transparency is used (non-zero) for the text in Inkscape, but the package 'transparent.sty' is not loaded}%
    \renewcommand\transparent[1]{}%
  }%
  \providecommand\rotatebox[2]{#2}%
  \ifx\svgwidth\undefined%
    \setlength{\unitlength}{454.76596069bp}%
    \ifx\svgscale\undefined%
      \relax%
    \else%
      \setlength{\unitlength}{\unitlength * \real{\svgscale}}%
    \fi%
  \else%
    \setlength{\unitlength}{\svgwidth}%
  \fi%
  \global\let\svgwidth\undefined%
  \global\let\svgscale\undefined%
  \makeatother%
  \begin{picture}(1,0.10908964)%
    \put(0,0){\includegraphics[width=\unitlength]{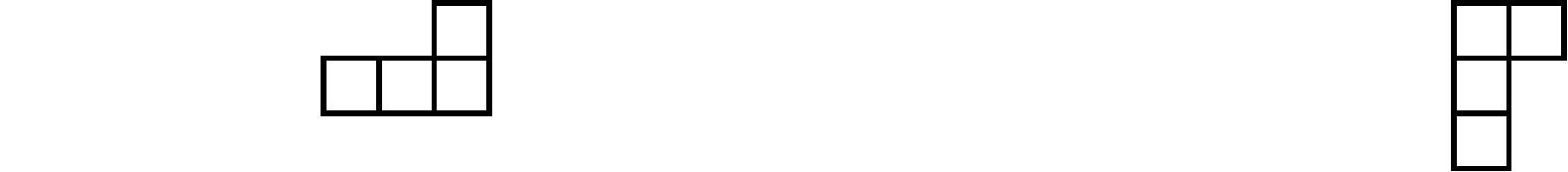}}%
    \put(-0.00045181,0.04852476){\color[rgb]{0,0,0}\makebox(0,0)[lb]{\smash{$\calg[2,3]=$}}}%
    \put(0.70320682,0.04852476){\color[rgb]{0,0,0}\makebox(0,0)[lb]{\smash{$\calg[1,1,3]=$}}}%
    \put(0.41814478,0.04869655){\color[rgb]{0,0,0}\makebox(0,0)[lb]{\smash{and}}}%
  \end{picture}%
\endgroup%
}}\]
 This reflection induces a bijection $\varphi$ between the sets of perfect matchings of the snake graphs. Under this bijection, the minimal matching of $\calg\cfa$ is mapped to the maximal matching of $\calg[1,a_1-1,a_2,\ldots,a_n]$, and this implies that for an arbitrary perfect matching $P$ of $\calg\cfa$ with height function $y(P)=y_{i_1}^{\za_1}\cdots y_{i_d}^{\za_d}$, the corresponding perfect matching $\varphi(P)$ of  $\calg[1,a_1-1,a_2,\ldots,a_n]$ has the complementary height function, that is, 
 \[y(\varphi(P))= y_{i_1}^{1-\za_1}y_{i_2}^{1-\za_2}\cdots y_{i_d}^{1-\za_d}=y_{i_1}y_{i_2}\cdots y_{i_d} \, \overline{y(P)}.\]
 This proves the first identity, and the second follows by specialization.
 \end{proof}

Our next lemma gives a recursive formula for the F-polynomial of an even continued fraction.  The result mainly follows  from results in \cite{CS2}, however, we are using even continued fractions here instead of positive continued fractions in loc.cit. 

\begin{lem}
 \label{lem 6} 
Let $\cfb$ be an even continued fraction.  
Then we have the following identity of $F$-polynomials depending on the type sequence of $\cfb$.

If $b_1>0$ then
 $F\cfb$ is equal to 
\[\left\{\begin{array}
 {ll}
  - F[b_1,b_2,\ldots,b_{m-2}] \prod_{\calg[b_m]} y_i\ +\ F[b_1,b_2,\ldots,b_{m-1}] F[b_m] &\textup{in type $(\ldots,-,-)$};\\ \\
 F[b_1,b_2,\ldots,b_{m-2}] \prod_{\calg\setminus\calg_2} y_i\ +\  F[b_1,b_2,\ldots,b_{m-1}] F[b_m] &\textup{in type $(\ldots,+,-)$};\\ \\   F[b_1,b_2,\ldots,b_{m-2}]\ +\ F[b_1,b_2,\ldots,b_{m-1}] F[b_m] y_\ell&\textup{in type $(\ldots,-,+)$};\\ \\
 - F[b_1,b_2,\ldots,b_{m-2}] \prod_{\calg_1\setminus\calg_2} y_i \ +\  F[b_1,b_2,\ldots,b_{m-1}] F[b_m]&\textup{in type $(\ldots,+,+)$},
\end{array}\right.
 \] 
 where
\begin{itemize}
\item []
 $\prod_{\calg[b_m]}$ runs over all tiles $G_i\in\calg[b_m]$; \\
 \item[] 
$\prod_{\calg\setminus\calg_2}$ runs over all tiles $G_i$ that lie in $\calg\cfb$ but not in $\calg[b_1,b_2,\ldots,b_{m-2}]$; \\
 \item[] 
 $y_\ell$ is the coefficient of the tile $G_\ell$ that connects $\calg[b_1,\ldots,b_{m-1}]$ and $\calg[b_m]$ in $\calg\cfb$; and\\
 \item[]
 $\prod_{\calg_1\setminus\calg_2}$ runs over all tiles $G_i$ that lie in $\calg[b_1,\ldots,b_{m-1}]$ but not in $\calg[b_1,b_2,\ldots,b_{m-2}]$.
\end{itemize}

\smallskip
If $b_1<0$, the above formulas hold if we replace the types with their negatives.
\end{lem}

\begin{proof}
 One reason for the separate cases  is that the construction of $\calg\cfb$ is different if $b_m$ and $b_{m-1}$ have the same sign or not, see section \ref{sect sgb}.
 Suppose first that $b_m$ and $b_{m-1}$ have the same sign, thus the type sequence has a sign change in the last two positions. 
 In this case, the grafting with a single edge formula \cite[section 3.3, case 3]{CS2} gives the following two results. 
 
 If the minimal matching $P_-$ restricts to a perfect matching of $\calg[b_m]$ then
 \[
\calg\cfb= \calg[b_1,b_2,\ldots,b_{m-2}] \prod_{\calg\setminus\calg_2} y_i \ + \ \calg[b_1,b_2,\ldots,b_{m-1}] \,\calg[b_m] .\]

 If the  $P_-$ does not restrict to a perfect matching of $\calg[b_m]$ then
 \[
\calg\cfb= \calg[b_1,b_2,\ldots,b_{m-2}]\ +\ \calg[b_1,b_2,\ldots,b_{m-1}]\, \calg[b_m]\,  y_\ell .\]

Suppose now that  $b_m$ and $b_{m-1}$ have opposite signs, thus the last two entries in the type sequence are equal. In this case,  \cite[section 2.5, case 2 \& Theorem 6.3]{CS} yields the following two results. 

If the $P_-$ restricts to a perfect matching of $\calg[b_m]$ then
 \[
\calg\cfb=-\calg[b_1,b_2,\ldots,b_{m-2}] \prod_{\calg_1\setminus\calg_2} y_i\ +\ \calg[b_1,b_2,\ldots,b_{m-1}] \,\calg[b_m] ,\]

If the $P_-$ does not restrict to a perfect matching of $\calg[b_m]$ then
 \[
\calg\cfb=-\calg[b_1,b_2,\ldots,b_{m-2}] \prod_{\calg[b_m]} y_i\ +\ \calg[b_1,b_2,\ldots,b_{m-1}]\, \calg[b_m] .\]

Thus we need to determine when $P_-$ restricts to a perfect matching of $\calg[b_m]$.  See Figure~\ref{match5} for examples that illustrate the argument. Suppose first that $b_1>0$. 
Then the first edge of the minimal matching $e_0$ has sign $f(e_0)=+$.  Moreover 
$e_0$ is a south edge, which implies that all south edges and all west edges of the minimal matching $P_-$ have sign $+$ and all north and all east edges in $P_-$ have sign $ -$, see \cite[Lemma 4.3]{CS2}. In particular, the last edge $e\in\calgNE$ of the minimal matching has sign $-$. 
On the other hand, the last edge  $e'_{\zb}\in\calgNE$ of the sign sequence of the even continued fraction $\cfb$ has sign $-$ if and only if $\cfb$ is of type $(\ldots,-)$. Thus
$e'_{\zb}$ is in the minimal matching $P_-$ if and only if the type sequence of $\cfb$ ends in a minus sign.

Since $\calg[b_m]$ is a zigzag graph, if $P_-$ contains $e'_{\zb}$, then it contains two boundary edges of the last tile and thus restricts to a matching of the last tile. Moreover, $P_-$ contains 3 boundary edges of the last two tiles and thus restricts to a matching of the last two tiles. Continuing this way, we see that $P_-$ restricts to a matching of the last $b_m-1$ tiles. The question whether $P_-$ restricts to a matching of the last $b_m$ tiles only depends on how $\calg[b_m]$ is glued to $\calg[b_1,\ldots,b_{m-1}]$. 
Thus, the minimal matching restricts to $\calg[b_m]$ in type $(\ldots,+,-)$; these cases are marked $(\ldots,+,-)$ in Figure \ref{match5}. On the other hand, the minimal matching does not restrict to $\calg[b_m]$ in type $(\ldots,-,-)$.

Now suppose $\cfb$ is of type $(\ldots,+)$. Then $e'_{\zb}$ is not in $P_-$. Thus $P_-$ contains only one edge of the last tile, only two edges of the last two tiles and only $b_{m}-2$ edges of the last $b_{m}-1 $ tiles. 
Again, the question whether $P_-$ restricts to a matching of the last $b_m$ tiles only depends on how $\calg[b_m]$ is glued to $\calg[b_1,\ldots,b_{m-1}]$.
Thus, the minimal matching restricts to $\calg[b_m]$ in type $(\ldots,+,+)$; these cases are marked $(\ldots,+,+)$ in Figure \ref{match5}. On the other hand, the minimal matching does not restrict to $\calg[b_m]$ in type $(\ldots,-,+)$.
This proves the statement in the case where $b_1>0$.

If $b_1<0$ the above argument is still valid, except that in this case, the sign of the edge $e_0$ is $-$, and therefore the roles of the signs are reversed. 
\end{proof}
\begin{figure}
\begin{center}
  {\scriptsize \scalebox{0.7}{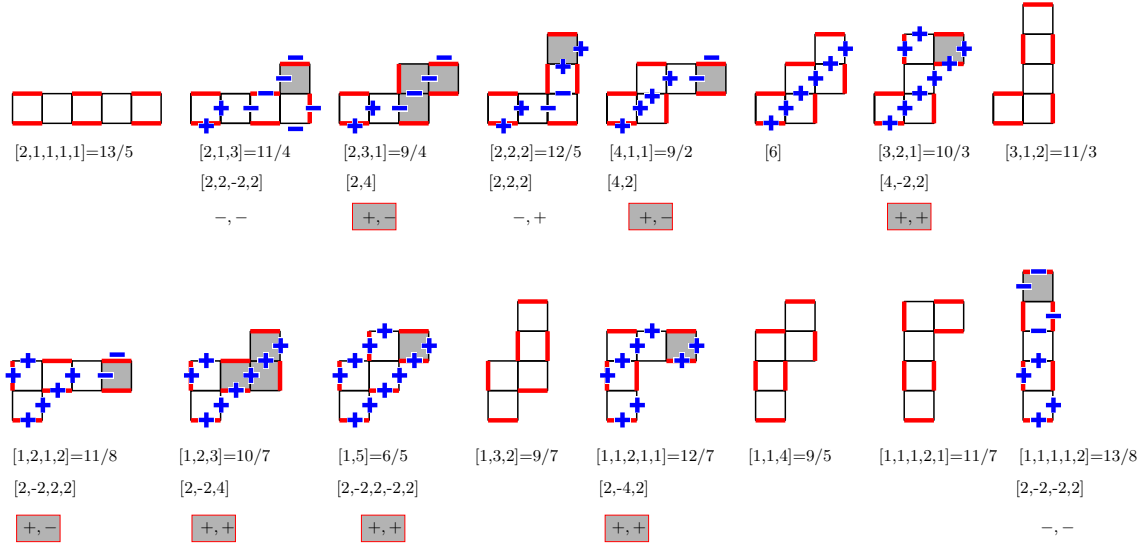}}
 \caption{A complete list of  snake graphs with 5  tiles together with their minimal matchings in bold red. If the snake graph corresponds to an even continued fraction $\cfb$ then the subgraph $\calg[b_m]$ is shaded. The labels indicate the last two entries in the type sequence of the continued fraction ($\pm,\pm$). This label sits in a shaded box if $P_-$ restricts to $\calg[b_m]$, which happens precisely when the label is $+,-$ or $+,+$.} 
 \label{match5}
\end{center}
\end{figure}

We now apply this result to the specialization of the $F$-polynomial.
\begin{cor}
 \label{cor 8bis} 
 Let $\cfb$ be an even continued fraction with $b_1>0$. Then
\[ \Fb= \mu(t)\,  F_{b_1,b_2,\ldots,b_{m-2}} +
\nu(t)\, [\,|b_m|\,]_q\,F_{b_1,b_2,\ldots,b_{m-1}} \]
 where the functions $\nu(t)$ and $\mu(t)$ depend on the type sequence of $\cfb$ as follows
%
 \begin{center}
  \begin{tabular}{  |c | c | r |}
\hline $\nu(t)$&$\mu(t)$&$\type[b_1,\ldots,b_m]$\\ \hline &&\\[-6pt]
$ 1$&$t^{- |b_m|+1} $&$(\ldots,-,-) $\\[4pt]$
 1$&$t^{- |b_m|-|b_{m-1}|} $&$(\ldots,-,+,-) $\\[4pt]$
 1$&$\ -t^{- |b_m|-|b_{m-1}|+1} \ $&$(\ldots,+,+,-) $\\[4pt]\hline &&\\[-6pt]
$\  -t^{-1}$\ &$1$&$(\ldots,-,+) $\\[4pt]$
 1$&$-t^{-|b_{m-1}|} $&$(\ldots,-,+,+) $\\[4pt]$
1$&$t^{-|b_{m-1}|+1} $&$(\ldots,+,+,+) $\\[4pt]
\hline
  \end{tabular}
\end{center}

\end{cor}
\begin{proof} According to  Lemma \ref{lem 6}, we only need to check that $\mu(t),\nu(t)$ are the correct functions. Moreover it is clear from the lemma that $\nu(t)=1$ in all cases except in the case $(\ldots,-,+)$ where $\nu(t)=-t^{-1}$. 

In type $(\ldots,-,-)$,  Lemma \ref{lem 6} implies that  $\mu(t)= - (-t^{-1})^\za$, where $\za $ is the number of tiles in $\calg[b_m]$. Since $b_m$ and $b_{m-1}$ have opposite signs we have $\za=|b_m|-1$, and since $b_m$ is even, we see that $\mu(t) =t^{-|b_m|+1}$. 

 In type $(\ldots,+,-)$,   Lemma \ref{lem 6} implies that $\mu(t)=
(-t^{-1})^{\za}$, where $\za$ is  the number of tiles in  $\calg\cfb \setminus\calg[b_1,b_2,\ldots,b_{m-2}]$. According to Corollary \ref{cor 3.3}, this number  is $|b_m|+|b_{m-1}| - \textup{(number of sign changes in $b_{m-2},b_{m-1})$}$. Thus $\za = |b_m|+|b_{m-1}|$ in type $(\ldots,-,+,-)$; and  $\za = |b_m|+|b_{m-1}|-1$ in type $(\ldots,+,+,-)$. Since $b_{m-1}$ and $b_{m-2}$ are both even, it follows that $\mu(t)= t^{- |b_m|-|b_{m-1}|}$ in type $(\ldots,-,+,-)$, and $\mu(t)= -t^{- |b_m|-|b_{m-1}|+1}$ in type $(\ldots,+,+,-)$.

In type $(\cdots,-,+)$, Lemma \ref{lem 6} implies  $\mu(t)=1$.

In type $(\cdots,+,+)$,
 Lemma \ref{lem 6} implies that  $\mu(t)=-(-t^{-1})^{\za}$, where $\za$ is  the number of tiles in  $\calg[b_1,b_2,\ldots,b_{m-1}] \setminus\calg[b_1,b_2,\ldots,b_{m-2}]$.  Similar to the previous case, we see that $\za = |b_{m-1}|$ in type $(\cdots,-,+,+)$ ; and  $\za = |b_{m-1}|-1$ in type $(\cdots,+,+,+)$, and this completes the proof.
\end{proof}

\section{Main results}\label{sect main}
In this section, we state and prove our main theorem. Recall that we have fixed the orientation of all knots and links in section \ref{sect knots}. 

\medskip

\begin{thm}
 \label{thm main}
 Let $\cfb$ be an even continued fraction, let $V_{\cfb}$ be the Jones polynomial of the corresponding 2-bridge link, and let $\Fb$ be the specialized $F$-polynomial of the corresponding cluster variable. Then
 \[ V_{\cfb}=  (- 1) ^{m-\tau} \,t^j \ \Fb \]
 where   $\tau$ is the number of subsequences $+,+$ in the type sequence of $\cfb$ and \[\textstyle j=  \sum_{i=1}^m  \max\left( (-1)^{i+1} b_i + \frac{\textup{sign}(b_ib_{i-1})}{2}\ ,\ -\frac{1}{2}\right).\]
\end{thm}

\begin{proof}
We proceed by induction on $m$. For $m=1$, we will show that $V_{[b_1]}= - t^jF_{b_1}$.
Suppose first that $b_1>0$. Then  $j=b_1+\frac{1}{2}$ and according to Definition \ref{def F} we have  $F_{b_1}= [b_1+1]_q-q$, with $q=-t^{-1}$.
On the other hand, Theorem \ref{thm 5} implies 
\[\begin{array}
 {rcl}
 V_{[b_1]}&=& t^{b_1}(-t^{-\frac{1}{2}}-t^{\frac{1}{2}})-t^{\frac{1}{2}}\,[b_1]_{\overline{q}}\\
 &=&-t^{b_1+\frac{1}{2}}(t^{-1}+1+t^{-b_1} (1-t+t^2-\cdots -t^{b_1-1}))\\
 &=&-t^{b_1+\frac{1}{2}}([b_1+1]_q-q),
\end{array}\]
and the result follows.

Now suppose that $b_1<0$. Then $j=-\frac{1}{2}$ and Definition \ref{def F} implies 
\[F_{b_1}= (-t^{-1})^{|b_1|} \overline{F_{|b_1|}}
=t^{-|b_1|}
[b_1+1]_{\overline{q}}\,-\overline{q}
= t^{-|b_1|}(1+t^2-t^3+t^4-\cdots+t^{|b_1|} )
= [|b_1|+1]_q-q^{|b_1|-1}.\]
On the other hand, Theorem \ref{thm 5} yields 
 \[\begin{array}
 {rcl}
 V_{[b_1]}&=& 
 t^{-|b_1|}(-t^{-\frac{1}{2}}-t^{\frac{1}{2}})-t^{-\frac{1}{2}}\,[\,|b_1|\,]_{{q}}\\
 &=&-t^{-\frac{1}{2}}(t^{-|b_1|}+t^{-|b_1|+1}+[\,|b_1|\,]_{{q}})\\
&=&-t^{-\frac{1}{2}}(\,[\,|b_1|+1]_q-q^{|b_1|-1}),
\end{array}\]
and the result follows.

Now suppose that $m>1$.
Replacing $\cfb$ with $[-b_1,-b_2,\ldots,-b_m]$ has the effect of replacing the link with its mirror image, and its Jones polynomial with its image under the bar involution, that is, $V_{[-b_1,-b_2,\ldots,-b_m]}=\overline{V_{\cfb}}$.
Therefore,  we may assume without loss of generality that $b_1>0$. 
If $m$ is even then $L$ is a knot and if $m$ is odd then $L$ is a 2-component link.

Now, from Theorem \ref{thm 5} we know that
 \[V_{\cfb}= 
 \left\{
\begin{array}
 {ll}
 t^{-|b_m|}\,V_{[b_1,b_2,\ldots,b_{m-2}]} -t^{-\frac{1}{2}}[\,|b_m|\,]_q\,V_{[b_1,b_2,\ldots,b_{m-1}]} 
&\textup{if $\type[b_1,\ldots,b_{m}]=(\ldots,-)$:}\\ \\
t^{|b_m|}\ V_{[b_1,b_2,\ldots,b_{m-2}]} -\ t^{\frac{1}{2}}\ [\,|b_m|\,]_{\overline {q}}\,V_{[b_1,b_2,\ldots,b_{m-1}]} 
&\textup{if $\type[b_1,\ldots,b_{m}]=(\ldots,+)$.}\\
\end{array}\right.
 \]
By induction, we have 
$V_{[b_1,b_2,\ldots,b_{m-i}]} =\zd_i \,t^{j_i}\, F_{b_1,b_2,\ldots,b_{m-i}}$, for $i=1,2.$ 
where $\zd_i\, t^{j_i}$ are the leading terms of the Jones polynomials and $\zd_i=\pm 1$. Therefore we get
\begin{equation}
\label{eq 73}
V_{\cfb}=\left\{ 
\begin{array} {ll}
 \zd_2 \,t^{j_2-|b_m|}\, F_{b_1,b_2,\ldots,b_{m-2}} 
 -\zd_1 \,t^{j_1-\frac{1}{2}}[\,|b_m|\,]_q\, F_{b_1,b_2,\ldots,b_{m-1}.} &\textup{in type $(\ldots,-)$;}\\ \\
 \zd_2 \,t^{j_2+|b_m|}\, F_{b_1,b_2,\ldots,b_{m-2}} 
 -\zd_1 \,t^{j_1+\frac{1}{2}}[\,|b_m|\,]_{\overline{q}}\, F_{b_1,b_2,\ldots,b_{m-1}.} &\textup{in type $(\ldots,+)$;}\\
\end{array}\right. 
\end{equation}

On the other hand, Corollary \ref{cor 8bis} implies that 
\begin{equation}\label{eq 74} \zd_0\, t^{j_0} \Fb= 
\zd_0\, t^{j_0}\,\mu(t)\,  F_{b_1,b_2,\ldots,b_{m-2}} +\zd_0\, t^{j_0}\,\nu(t)\,[\,|b_m|\,]_q\, F_{b_1,b_2,\ldots,b_{m-1}} \end{equation}
 where $\zd_0\, t^{j_0}$ is the leading term of $V_{\cfa}$ and the functions $\mu(t),\nu(t)$ depend on the type sequence of $\cfb$ as listed in the table in Corollary \ref{cor 8bis}. 
 We want to show that the expressions on the left hand side of equations (\ref{eq 73}) and (\ref{eq 74}) are equal, so it suffices to show   the equality of the expressions on the right hand side. To do so, we have to go through the different cases and show that the expressions in front of the $F$-polynomials are equal.
This is done using the tables of Corollary \ref{cor 8bis} (for $\nu(t)$ and $\mu(t)$),
Corollary \ref{cor 7} (comparing $j_0,j_1$ and $j_2$) and Corollary \ref{cor 8} (comparing $\zd_0$, $\zd_1$ and $\zd_2$).

In type $(\ldots,-,-)$ we have
$\nu(t)=1$, $\mu(t)=t^{-|b_m|+1}$, $j_0=j_1-1/2 =j_2-1$, and
$\zd_0=-\zd_1=\zd_2$. Therefore 
\[\zd_0\, t^{j_0}\,\mu(t) = \zd_2 \,t^{j_2-|b_m|} \textup{ and }\zd_0\, t^{j_0}\, \nu(t) = -\zd_1\, t^{j_1-\frac{1}{2}} \] 
as desired.

In type $(\ldots,-,+,-)$ we have
$\nu(t)=1$, $\mu(t)=t^{-|b_m|-|b_{m-1}|}$, $j_0=j_1-1/2 =j_2+|b_{m-1}|$, and
$\zd_0=-\zd_1=\zd_2$. Again
\[\zd_0\, t^{j_0}\,\mu(t) = \zd_2 \,t^{j_2-|b_m|} \textup{ and }\zd_0\, t^{j_0}\, \nu(t) = -\zd_1\, t^{j_1-\frac{1}{2}}. \] 

In type $(\ldots,+,+,-)$ we have
$\nu(t)=1$, $\mu(t)=-t^{-|b_m|-|b_{m-1}|+1}$, $j_0=j_1-1/2 =j_2+|b_{m-1}|-1$, and
$\zd_0=-\zd_1=-\zd_2$. Again
\[\zd_0\, t^{j_0}\,\mu(t) = \zd_2 \,t^{j_2-|b_m|} \textup{ and }\zd_0\, t^{j_0}\, \nu(t) = -\zd_1\, t^{j_1-\frac{1}{2}}. \] 

Now in type $(\ldots,-,+)$ we have
$\nu(t)=-t^{-1}$, $\mu(t)=1$, $j_0=j_1+|b_m|+1/2 =j_2+|b_m|$, and
$\zd_0=-\zd_1=\zd_2$. Therefore 
\[\zd_0\, t^{j_0}\,\mu(t) = \zd_2 \,t^{j_2+|b_m|} \] 
which agrees with the first term on the right hand side of equation (\ref{eq 73}). For the second term we get \[\zd_0\, t^{j_0}\, \nu(t) \,[\,|b_m|\,]_q = \zd_1\, t^{j_1+|b_m|-\frac{1}{2}}\,[\,|b_m|\,]_q = -\zd_1\, t^{j_1+\frac{1}{2}}\,[\,|b_m|\,]_{\overline{q}},  \] where the last equation follows from $ -t^{|b_m|-1}\,[\,|b_m|\,]_q =[\,|b_m|\,]_{\overline{q}}$. This expression agrees with the second term in  (\ref{eq 73}).

In type $(\ldots,-,+,+)$ we have
$\nu(t)=1$, $\mu(t)=-t^{-|b_{m-1}|}$, $j_0=j_1+|b_m|-1/2 =j_2+|b_m|+ |b_{m-1}|$, and
$\zd_0=\zd_1=-\zd_2$. Again
\[\zd_0\, t^{j_0}\,\mu(t) = \zd_2 \,t^{j_2+|b_m|} \textup{ and }\zd_0\, t^{j_0}\, \nu(t) \,[\,|b_m|\,]_q= \zd_1\, t^{j_1+|b_m|-\frac{1}{2}}  \,[\,|b_m|\,]_q
=- \zd_1\, t^{j_1+\frac{1}{2}}  \,[\,|b_m|\,]_{\overline{q}} \] 
as desired.

Finally, in type $(\ldots,+,+,+)$ we have
$\nu(t)=1$, $\mu(t)=t^{-|b_{m-1}|+1}$, $j_0=j_1+|b_m|-1/2 =j_2+|b_m|+ |b_{m-1}|-1$, and
$\zd_0=\zd_1=\zd_2$. Again we have
\[\zd_0\, t^{j_0}\,\mu(t) = \zd_2 \,t^{j_2+|b_m|} \textup{ and }\zd_0\, t^{j_0}\, \nu(t) \,[\,|b_m|\,]_q= \zd_1\, t^{j_1+|b_m|-\frac{1}{2}}  \,[\,|b_m|\,]_q
=- \zd_1\, t^{j_1+\frac{1}{2}}  \,[\,|b_m|\,]_{\overline{q}}. \] 
This proves the identity in the theorem. Since, by Lemma \ref{lemF}, the degree of $\Fb$ is zero, we see that $j$ is the degree of the Jones polynomial. The explicit formulas for $j$   and for the sign $ (-1)^{m-\tau}$  in the statement were already proved in Theorem \ref{thm degree}. 
\end{proof}

 We say that Laurent polynomial in $t$ is {\em alternating} if it can be written as $\pm\sum_{i=m}^M (-1)^i a_i t^i$ with $m,M\in \mathbb{Z}$ and $a_i\in \mathbb{Z}_{\ge 0}$.
The \emph{width} of a Laurent polynomial is the difference between the highest and the lowest exponent. For example $V_{[3]}=t^{-1}+t^{-3}-t^{-4}$ is alternating and has width $-1-(-4)=3$. 
\begin{cor}\label{cor j}
 The Jones polynomial $V_{\cfa}$ is an alternating sum, its width is $a_1+a_2+\cdots+a_n$ and the first and the last coefficient has absolute value 1.
\end{cor}
\begin{proof}
 This follows from Theorem \ref{thm main} using Lemma \ref{lemF}.
\end{proof}

\begin{remark}
For arbitrary knots, the Jones polynomial is not necessary alternating. For example, the Jones polynomial of  the knot  $10_{124}$ from Rolfsen's table is
$-q^{10}+q^6+q^4$.
\end{remark}

Combining Theorem \ref{thm main} and Proposition \ref{prop F}, we obtain the following direct formula for the Jones polynomial.
\begin{thm}\label{thm main2} Let $\cfa$ be a positive continued fraction. Then up to normalization by its leading term, $V_{\cfa}$ is equal to the numerator of the following continued fraction 

\begin{itemize}
\item [{\rm(a)}] If $n$ is odd, 
 \[\Big[\, [a_1+1]_q -q\ ,\  [a_2]_q \,q^{-\ell_2} ,\ [a_3]_q\,q^{\ell_2+1} ,\ldots, [a_{2i}]_q\, q^{-\ell_{2i}},\ [a_{2i+1}]_q\, q^{\ell_{2i}+1} ,\ldots , \ [a_{n}]_q\, q^{\ell_{n-1}+1}\Big]. \]
 
 
\item [{\rm(b)}] If $n$ is even, 
\[q^{\ell_n}\,\Big[\, [a_1+1]_q -q\ ,\  [a_2]_q \,q^{-\ell_2} ,\ [a_3]_q\,q^{\ell_2+1} ,\ldots, [a_{2i}]_q\, q^{-\ell_{2i}},\ [a_{2i+1}]_q\, q^{\ell_{2i}+1} ,\ldots , \ [a_{n}]_q\, q^{-\ell_{n}}\Big]. \]
\end{itemize}
\end{thm}

 
\begin{cor}
If the continued fraction has only one coefficient, we have 
\begin{eqnarray}
 V_{[2a+1]}& =& 
 \ t^{-a} \ (1+t^{-2}-t^{-3}+t^{-4}\cdots-t^{-2a-1}) ,\\
 V_{[-2a]}  
 &=&  -t^{ -\frac{1}{2}}\ (1+t^{-2}-t^{-3}+t^{-4}\cdots-t^{-2a-1}), \\
 V_{[2a]} 
 &=& -t^{ 2a+\frac{1}{2}}  \,
 (1+t^{-2}-t^{-3}+t^{-4}\cdots-t^{-2a-1}).
\end{eqnarray}

\end{cor}
\begin{proof}
 This follows immediately from Theorem \ref{thm main} and Proposition \ref{prop F}.
\end{proof}

\subsection{Coefficients of the Jones polynomial}
Our realization of the Jones polynomial in terms of snake graphs allows us to compute the coefficients directly. We give here  formulas for the first three and the last three coefficients.

 Let $\cfa$ be a positive continued fraction with $a_1\ge 2$ and $a_n\ge 2$,  and let $\ell_i=a_1+a_2+\cdots +a_i$. By Corollary \ref{cor j},  the Jones polynomial is of the form
 \[V_{\cfa}= \pm\, t^j\,\sum_{i=0}^{\ell_n} (-1)^i\, v_i\, t^i ,\] 
 with $v_i\in \mathbb{Z}_{\ge 0}$. 
 We denote by $\za$  the cardinality of the set 
 $\{i \mid a_i =1\}$, and, using the Kronecker delta notation,  we let $\zd_{a_i,2}= 1$ if $a_i=2$ and $\zd_{a_i,2}= 0$ if $a_i\ne2$

\begin{thm}\label{thm coeff} We have the following formulas for the coefficients of the Jones polynomial
 \[
\begin{array}
 {rclcrcl}
 v_0 &=& 1 \\v_{\ell_n} &=& 1 \\
\\
 v_1 &=& \left\{\begin{array}{ll} 
 k  &\textup{if $n=2k+1$;} \\[5pt]
 k &\textup{if $n=2k$.}  \end{array}\right. 
\\ \\
  v_{\ell_n-1} &=& \left\{\begin{array}{ll} 
 k+1  &\textup{if $n=2k+1$;} \\[5 pt]
 k &\textup{if $n=2k$.}\\  \end{array}\right.\\ 
 \\
v_2&=& \left\{\begin{array}{ll} \frac{1}{2} (k+1)(k+2) -\za&\textup{if $n=2k+1$;} \\[5pt]
\frac{1}{2} \,k(k+3)-\za &\textup{if $n=2k$.}
 \end{array}\right.
 \\ \\
 v_{\ell_n-2}&=& \left\{\begin{array}{ll} \frac{1}{2} (k^2+5k+2)-\za-\zd_{a_1,2}-\zd_{a_n,2} &\textup{if $n=2k+1$;} \\[5pt]
\frac{1}{2}\, k(k+3) -\za-\zd_{a_1,2} &\textup{if $n=2k$.}
 \end{array}\right.\\
\end{array}\]
\end{thm}

\begin{proof}
 According to Theorem \ref{thm main}, the coefficients $v_i$ are the same as the coefficients of the specialized $F$-polynomial, which in turn are determined by the continued fraction formula in Proposition \ref{prop F}. Suppose first that $n=2k+1$ is odd. Then we must compute the coefficients of 
\begin{equation}
\label{eq 71}
 N\Big[\, [a_1+1]_q -q\ ,\  [a_2]_q \,q^{-\ell_2} ,\ldots, [a_{2i}]_q\, q^{-\ell_{2i}},\ [a_{2i+1}]_q\, q^{\ell_{2i}+1} ,\ldots , \ [a_{n}]_q\, q^{\ell_{n-1}+1}\Big]  
\end{equation}
 which by definition (\ref{def N}) is equal to 
 \begin{eqnarray} \label{eq 71a}
 & [a_{n}]_q\, q^{\ell_{n-1}+1} \  N\Big[\, [a_1+1]_q -q\ ,\  [a_2]_q \,q^{-\ell_2} ,\ldots, \ [a_{n-1}]_q\, q^{-\ell_{n-1}}\Big]
\\
 +&N\Big[\, [a_1+1]_q -q\ ,\  [a_2]_q \,q^{-\ell_2} ,\ldots, \ [a_{n-1}]_q\, q^{\ell_{n-2}+1}\Big].\nonumber\end{eqnarray}
Using parts (a) and (b) of Proposition \ref{prop F}, we see that (\ref{eq 71a})  is equal to  
\begin{equation}
 \label{eq 72}
[ a_{n}]_q\, q\,  F_{a_1,\ldots, a_{n-1}}+ F_{a_1,\ldots, a_{n-2}}.
\end{equation}
 We let $v_i'$ and $ v_i''$ denote the coefficients of $F_{a_1,\ldots, a_{n-1}}$ and $F_{a_1,\ldots, a_{n-2}}$ respectively, and  recall that $\ell_{n-1}$ and $\ell_{n-2}$ are their degrees.
With this notation (\ref{eq 72})  becomes
\[\begin{array}{ll}[a_n]_q\,q\,(1+v_1'\,q+v_2'\,q^2+\ldots+v_{\ell_{n-1}-2}''\,q^{\ell_{n-1}-2}+v_{\ell_{n-1}-1}'\,q^{\ell_{n-1}-1}+q^{\ell_{n-1}})\\[5pt]
+1+v_1''\,q+v_2''q^2 +\ldots+ v_{\ell_{n-2}-2}''\, q^{\ell_{n-2}-2}+ v_{\ell_{n-2}-1}''\, q^{\ell_{n-2}-1}+q^{\ell_{n-2}}
\end{array}\]
and this shows that
$v_0=v_{\ell_n}=1$, and
\[ 
\begin{array}
 {rclcrcl}
 v_1&=&1+v_1'' &\qquad& v_2&=&1+v_1' +v_2''\\
 v_{\ell_n-1}&=& 1+v_{\ell_{n-1}-1}' && v_{\ell_n-2}&=& v_{\ell_{n-1}-2}'+v_{\ell_{n-1}-1}'+1-\zd_{a_n,2}.
\end{array}
\]
Applying our induction hypothesis, we get
\[ 
\begin{array}
 {rclcrcl}
 v_1&=&1+(n-3)/2=k &\qquad& v_2&=&1+k +k(k+1)/2-\za\\
 v_{\ell_n-1}&=& 1+k && v_{\ell_n-2}&=& k(k+3)/2-\za-\zd_{a_1,2}+k+1-\zd_{a_n,2}.
\end{array}
\]
This implies the result for $v_1$ and $v_{\ell_n-1}$ directly, and for $v_2$ it follows from the computation
$1+k +k(k+1)/2 =(k^2+3k+2)/2=(k+1)(k+2)/2$, and for $v_{\ell_n-2}$ it follows from the computation
$1+k +k(k+3)/2 =(k^2+5k+2)/2$.

Now suppose $n=2k$ is even. In this case,  we must compute the coefficients of 
\begin{equation}
\label{eq 77}
q^{\ell_n}\, N\Big[\, [a_1+1]_q -q\ ,\  [a_2]_q \,q^{-\ell_2} ,\ldots, [a_{2i}]_q\, q^{-\ell_{2i}},\ [a_{2i+1}]_q\, q^{\ell_{2i}+1} ,\ldots , \ [a_{n}]_q\, q^{-\ell_{n}}\Big]  
\end{equation}
 which by definition is equal to 
 \begin{eqnarray}\label{eq 77a}
 & [a_{n}]_q \,  N\Big[\, [a_1+1]_q -q\ ,\  [a_2]_q \,q^{-\ell_2} ,\ldots, \ [a_{n-1}]_q\, q^{\ell_{n-1}+1}\Big]
\\
 +&q^{\ell_n}\,N\Big[\, [a_1+1]_q -q\ ,\  [a_2]_q \,q^{-\ell_2} ,\ldots, \ [a_{n-1}]_q\, q^{-\ell_{n-2}}\Big].\nonumber\end{eqnarray}
Again using parts (a) and (b) of Proposition \ref{prop F}, we see that (\ref{eq 77a})  is equal to  
\begin{equation}
 \label{eq 78}
[ a_{n}]_q\  F_{a_1,\ldots, a_{n-1}}+ q^{\ell_n-\ell_{n-2}}\,F_{a_1,\ldots, a_{n-2}},
\end{equation}
and since $\ell_n-\ell_{n-2}=a_n+a_{n-1}$, using the same notation as before, this becomes
\[\begin{array}{ll}
[a_n]_q\ 
(1+v_1'\,q+v_2'\,q^2+\ldots+v_{\ell_{n-1}-2}''\,q^{\ell_{n-1}-2}+v_{\ell_{n-1}-1}'\,q^{\ell_{n-1}-1}+q^{\ell_{n-1}})\\[5pt]
+ q^{a_n+a_{n-1}}(1+v_1''\,q+v_2''\,q^2 \ldots+ v_{\ell_{n-2}-2}''\, q^{\ell_{n-2}-2}+ v_{\ell_{n-2}-1}''\, q^{\ell_{n-2}-1}+q^{\ell_{n-2}}).
\end{array}\]
This shows that
$v_0=v_{\ell_n}=1$, and, since $\ell_{n}=\ell_{n-1}+a_n=\ell_{n-2}+a_n+a_{n-1}$,
\[ 
\begin{array}
 {rclcrcl}
 v_1&=&1+v_1' &\qquad& v_2&=&1+v_1' +v_2'\\
 v_{\ell_n-1}&=& 1+v_{\ell_{n-2}-1}'' && v_{\ell_n-2}&=&1+ v_{\ell_{n-1}-1}'+v_{\ell_{n-2}-2}''.
\end{array}
\]
Applying our induction hypothesis, we get
\[ 
\begin{array}
 {rclcrcl}
 v_1&=&1+k-1=k &\qquad& v_2&=&1+k-1 +k(k+1)/2-\za\\
 v_{\ell_n-1}&=& 1+k-1=k && v_{\ell_n-2}&=& 1+ k+ (k-1)(k+2)/2-\za-\zd_{a_1,2}.
\end{array}
\]
This implies the result for $v_1$ and $v_{\ell_n-1}$ directly, and for $v_2$ it follows from the computation
$k +k(k+1)/2=(k^2+3k)/2=k(k+3)/2$, and for $v_{\ell_n-2}$ it follows from the computation
$ 1+ k+ (k-1)(k+2)/2=(k^2+3k)/2$.
\end{proof}

\begin{example} Suppose the continued fraction $\cfa$ is such that all $a_i\ge 2$ and $a_1, a_n\ge 3$. Thus $\za=\zd_{a_1,2}=\zd_{a_n,2}=0$.
 For these cases the  values of the coefficients for various $n$   are listed in the following table.
 \smallskip
 
 \begin{center}
  \begin{tabular}{  |c ||c|c|c|c| c|c | c |}
\hline $n$&$\ v_0\ $&$\ v_1\ $&$\ v_2\ $&$\cdots$&$v_{\ell_n-2}$&$ v_{\ell_n-1} $&$\, v_{\ell_n}\, $\\ \hline 
1 &1 & 0&1&&1&1&1\\\hline
2&1&1&2&&2&1&1\\\hline
3&1&1&3&&4&2&1\\\hline
4&1&2&5&&5&2&1\\\hline
5&1&2&6&&8&3&1\\\hline
6&1&3&9&&9&3&1\\\hline
7&1&3&10&&13 &4&1\\\hline
8&1&4&14&&14 &4&1\\\hline
9&1&4&15&&19 &5&1\\\hline
10&1&5&20&&20 &5&1\\\hline
\hline
  \end{tabular}
\end{center}
\end{example}

Using a result from \cite{FKP}, which gives bounds for the volume of  a knot in terms of $v_1 $ and $v_{\ell_n-1}$, we have the following corollary.
\begin{cor}
 \label{cor vol} Let $C=C\cfa$ be a 2-bridge link such that $a_i\ge3$ for all $i$. Then the volume of its complement is bounded as follows.
 \[0.35367 (n-2) < \textup{vol}(S^3\setminus C) < 30 v_3 (n-1).\]
 where $v_3\sim 1.0149$ is the volume of a regular ideal tetrahedron.
\end{cor}
\begin{proof}
According to  Corollary 1.6 in \cite{FKP}, we only need to show that $v_1+v_{\ell_n-1}=n$. This follows immediately from Theorem \ref{thm coeff}. 
\end{proof}

\subsection{Examples}

\subsubsection{Hopf link}
$C[2] $  is the Hopf link. It has two components and the Jones polynomial depends on the orientation as explained in section \ref{sect Jones}. We have 
\[V_{[-2]}=-t^{-1/2}-t^{-5/2}= -t^{-1/2}(1+t^{-2}) \qquad\textup{ or }\qquad V_{[2]}= -t^{5/2}-t^{1/2} =-t^{5/2}(1+t^{-2}).\]
In both cases, the quotient of the  Jones polynomial  by its leading term is equal to \[F_2 = [3]_q -q =1+q^2=1+t^{-2}\qquad\textup{ and }\qquad F_{-2}=t^{-2}\,\overline{F_2}=1+t^{-2}.\]

\subsubsection{Trefoil}
 $C[3]=C[-2,2]$ is the trefoil knot, and we have 
\[V_{[-2,2]}= t^{-1}+t^{-3}-t^{-4} =t^{-1} (1+t^{-2}-t^{-3}) = t^{-1} F_{-2,2},\] 
where the last identity follows from Example \ref{ex 6.4}. 

\subsubsection{4 crossings}If the number of crossings is 4 we have either the figure 8 knot $C[2,2]$ or the link $C[4]$. In the first case, we have
\[V_{[2,2]}= t^2-t+1-t^{-1}+t^{-2}= t^2(1-t^{-1}+t^{-2}-t^{-3}+t^{-4}),\]
whereas
\[F_{2,2}=q^4 N\big[\,[3]_q-q,[2]_q\,q^{-4}\big] =  q^4 ((1+q^2)(1+q)q^{-4}+1) =1+q+q^2+q^3+q^4.
\]

On the other hand, for $C[4]$, we have seen in Examples \ref{ex 4} and \ref{ex 6.5} that 
\[ V_{[4]}=-t^{9/2}-t^{5/2}+t^{3/2}-t^{1/2}=-t^{9/2}(1 +t^{-2}-t^{-3}+t^{-4})=-t^{9/2}F_4,\] 
and
\[ V_{[-4]}=\overline{ V_{[4]}}=-\overline{t^{9/2}}\ \overline{(1 +t^{-2}-t^{-3}+t^{-4})}=-\overline{t^{9/2}}\ \overline{F_4}=-t^{-1/2}F_{-4}.\]

\subsubsection{Large examples}
We present two large examples to demonstrate the efficiency of the continued fraction formula. For the continued fractions we shall use the Euler-Minding formula, see \cite[p.9]{Perron}, which
states that $N\cfa$ is equal to
\[a_1a_2\ldots a_n \left(1 + \sum_{i=1}^{n-1} \frac{1}{a_ia_{i+1}}   + \sum_{i<k-1}^{n-2} \frac{1}{a_ia_{i+1}} \frac{1}{a_ka_{k+1}}   + \sum_{i<k-1<\ell-2}^{n-2} \frac{1}{a_ia_{i+1} }\frac{1}{a_ka_{k+1} }\frac{1}{a_\ell a_{\ell+1}}+\cdots\right)  .\]
Thus the first term is the product of all entries of the continued fraction, then in the first sum we remove all possible consecutive pairs from this term, in the second sum we remove two such pairs and so on. 
For example, $N[a,b,c,d,e]=abcde+cde+ade+abe+abc+e+c+a$.

\begin{itemize}
\item[(i)]
{$C[3,2,4]$} To compute the Jones polynomial of $C[3,2,4]$, a knot with 9 crossings, we need to calculate $F_{3,2,4}$ which is given by the numerator of the continued fraction
$[\,[4]_q-q, [2]_q\,q^{-5}, [4]_q\,q^6\,]$, and by the Euler-Minding formula this is equal to 
\[
\begin{array}
 {rcl}
 &([4]_q-q)  [2]_q\,q^{-5} [4]_q\,q^6 + ([4]_q-q)+ [4]_q\,q^6\\ [5pt]
 =&(1+q^2+q^3)(1+q)(1+q+q^2+q^3)q+(1+q^2+q^3)+ (1+q+q^2+q^3)q^6 \\[5pt]
 =&q^9 + 2 q^8 + 4 q^7 + 5 q^6 + 5 q^5 + 5 q^4 + 4 q^3 + 3 q^2 + q + 1
\end{array}
\]
Thus the Jones polynomial is 
\[V_{[3,2,4]}= {\pm} \,t^j (-t^{-9} + 2 t^{-8} - 4 t^{-7} + 5 t^{-6} - 5 t^{-5} + 5 t^{-4} - 4 t^{-3} + 3 t^{-2} - t^{-1} + 1).\]

\item[(ii)]{$C[2,3,4,5,6]$} This is a 2 component link with 20 crossings. The continued fraction has value $[2,3,4,5,6]=972/421 $.
Again using the Euler-Minding formula we see that $F_{2,3,4,5,6}$ is equal to 
\[
\begin{array}
 {ll}
 &([3]_q-q)  \,[3]_q\,q^{-5}\, [4]_q\,q^6 \,[5]_q\,q^{-14} \,[6]_q\,q^{15} 
 + [4]_q\,q^6 \,[5]_q\,q^{-14} \,[6]_q\,q^{15} +([3]_q-q) \,[5]_q\,q^{-14} \,[6]_q\,q^{15} \\ [5pt]
 &+([3]_q-q)  \,[3]_q\,q^{-5}\,[6]_q\,q^{15}+([3]_q-q)  \,[3]_q\,q^{-5}\, [4]_q\,q^6 
+[6]_q\,q^{15} + [4]_q\,q^6+([3]_q-q)\\
\\
=&q^2 
(1+q^2)  
(1+q+q^2)
(1+q+q^2+q^3)
(1+q+q^2+q^3+q^4)
(1+q+q^2+q^3+q^4+q^5) \\[5pt]
&+q^7
(1+q+q^2+q^3)
(1+q+q^2+q^3+q^4)
(1+q+q^2+q^3+q^4+q^5)
\\[5pt]
& +q
(1+q^2)  
(1+q+q^2+q^3+q^4)
(1+q+q^2+q^3+q^4+q^5)\\[5pt]
&+q^{10} 
(1+q^2)  
(1+q+q^2)
(1+q+q^2+q^3+q^4+q^5)\\[5pt]
&+q 
(1+q^2)  
(1+q+q^2)
(1+q+q^2+q^3)
\\[5pt]
&+q^{15}
(1+q+q^2+q^3+q^4+q^5) 
+q^6(1+q+q^2+q^3)
+(1+q^2)  
\\ \\
=&  t^{-20} - 3t^{-19} + 7t^{-18} - 15t^{-17} + 27t^{-16} - 44t^{-15} + 63t^{-14} 
- 83t^{-13} + 101t^{-12 }- 111t^{-11} \\&+ 113t^{-10} - 106t^{-9} + 92t^{-8} - 
73t^{-7} + 54t^{-6} - 36t^{-5} + 22t^{-4} - 12t^{-3} + 6t^{-2} - 2t^{-1}+1.
\end{array}
\]

\end{itemize}

{}
 \end{document}